\documentclass[12pt,oneside,reqno]{amsart}

\usepackage{amsmath, amsthm, amsfonts, amsfonts,amssymb,amscd,amsmath,latexsym,enumerate,verbatim, calc, quotmark, mathtools, txfonts, calrsfs, imakeidx, fancyhdr, listings, lipsum, tkz-euclide, amsthm}
\usepackage{cite}
\usepackage[pagewise]{lineno}
\usepackage{mathtools}

\usepackage[T1]{fontenc}
\usepackage[utf8]{inputenc}
\usepackage[a4paper, total={6in, 8in}]{geometry}
\usepackage{datetime2}
\usepackage[yyyymmdd,hhmmss]{datetime}
\usepackage[all, cmtip]{xy}

\usepackage[pagebackref]{hyperref}
\definecolor{deepblue}{rgb}{0,0,1}
\hypersetup{
     colorlinks = true, linkcolor = deepblue,
     citecolor   = deepblue,
     urlcolor    = blue,
}

\def\NZQ{\mathbb}               
\def\NN{{\NZQ N}}

\def\ZZ{{\NZQ Z}}
\def\Z{{\NZQ Z}}
\def\RR{{\NZQ R}}

\def\m{\mathfrak{m}}
\def\n{\mathfrak{n}}
\newcommand{\N}{\mathbb{N}}
\DeclareMathOperator{\Um}{Um}

\DeclareMathOperator{\E}{E}

\DeclareMathOperator{\W}{W}
\DeclareMathOperator{\NW}{NW}
\DeclareMathOperator{\Imj}{Im}
\DeclareMathOperator{\rad}{rad}

\def\v{\mathbf{v}}
\def\u{\mathbf{u}}
\def\w{\mathbf{w}}
\def\i{\mathfrak{i}}
\def\j{\mathfrak{j}}
\def\e{\varepsilon}
\def\p{\mathfrak{p}}

\def\GL{{\operatorname{GL}}}          
\def\M{{\operatorname{M}}}            
\def\SL{{\operatorname{SL}}}          
\def\vpmod{\!\pmod} 
\DeclareMathOperator*{\Spec}{Spec}
\DeclareMathOperator*{\maxSpec}{Max}
\DeclareMathOperator*{\height}{ht}

\DeclareMathOperator*{\Nil}{Nil}

\newtheorem{innercustomgeneric}{\customgenericname}
\providecommand{\customgenericname}{}
\newcommand{\newcustomtheorem}[2]{%
  \newenvironment{#1}[1]
  {%
   \renewcommand\customgenericname{#2}%
   \renewcommand\theinnercustomgeneric{##1}%
   \innercustomgeneric
  }
  {\endinnercustomgeneric}
}

\newcustomtheorem{customthm}{Theorem}
\newcustomtheorem{customqn}{Question}

\newtheorem{theorem}{Theorem}[section]
\newtheorem{lemma}[theorem]{Lemma}
\newtheorem{corollary}[theorem]{Corollary}
\newtheorem{proposition}[theorem]{Proposition}
\newtheorem{remark}[theorem]{Remark}

\newtheorem{example}[theorem]{Example}

\newtheorem{definition}[theorem]{Definition}

\newtheorem{question}[theorem]{Question}

\newtheorem*{acknowledgement}{Acknowledgement}

\title{A generalization of Rao's theorem to graded $R$-subalgebras of $R[t]$}

\date{\today, \currenttime}

\author{Diksha Garg, Anjan Gupta}

\address{Department of Mathematics\\
Indian Institute of Science Education and Research Bhopal\\
Bhopal Bypass Road, Bhopal, Madhya Pradesh, India. Pin - 462066.}
\email{dgarg1910401@gmail.com, anjan@iiserb.ac.in}

\thanks{Corresponding author: Diksha Garg; 
{\it email: dgarg1910401@gmail.com}}

\begin{document}
 
\begin{abstract}
Let $R$ be a Noetherian local ring of Krull dimension $d$ such that $(d!)R = R$, and let $A$ be a graded $R$-subalgebra of the polynomial algebra $R[t]$. We prove that every unimodular row of length $d + 1$ over $A$ can be completed to an invertible matrix. This is a generalization of a classical result by Rao, who proved that in the same setting, every unimodular row of length $d + 1$ over $R[t]$ admits a completion to an invertible matrix.
\end{abstract}
\maketitle
 
\noindent {\it Mathematics Subject Classification: 13A02, 13B25, 13C10, 19A13, 19B10, 19B14}

\noindent {\it Keywords: Graded algebras, Unimodular elements, Symplectic Witt groups, Invertible matrices,  Projective modules}
\section{introduction}
Let $R$ be a commutative Noetherian ring with unity and $R[t]$ denote the polynomial algebra over $R$ in the indeterminate $t$. The study of projective modules and lower $K$-groups within algebraic $K$-theory fundamentally relies on the notion of unimodular rows.  A row vector $\v \in R^n$ is called a \emph{unimodular row} of length $n$ over $R$ if its coordinates generate the unit ideal $R$. The first row of any invertible matrix is unimodular. The converse fails to hold in general, as illustrated by the unimodular row $(x_0, \ldots, x_{d})$ over the coordinate ring $\frac{\RR[X_0, \ldots, X_{d}]}{(X_0^2 + \ldots + X_{d}^2 -1)}$ of the real $d$-sphere, which is not completable to an invertible matrix unless $d = 1,3,7$. Here, $x_i$ denotes the residue class of the variable $X_i$. A fundamental problem in classical algebraic $K$-theory pertains to the conditions under which a unimodular row can be realized as the first row of an invertible matrix. This problem, often referred to as the \emph{completion problem} for unimodular rows, has garnered significant attention and inspired much optimism, particularly in the wake of a seminal question formulated by Suslin \cite[\S 5]{suslin1977stably}, which we now state in an equivalent form tailored to our purposes for $n\in \NN$, $n\geq 2$.

\begin{customqn}{\rm $Q(n)$}\label{Q1}
 Let $R$ be a ring such that $(n!) R = R$. and $\v(t)$ be a unimodular row of length
$n + 1$ over $R[t]$. Does there exist a $\sigma\in \SL_{n+1}(R[t])$ such that $\v(t) \sigma = \v(0)$?
\end{customqn}
When $R$ is local, the question is tantamount to asking if $\v(t)$ is completable to a matrix in $\SL_{n+1}(R[t])$. 
Now, we assume further that the dimension of $R$ is $d$. In 1977, Suslin \cite{suslin1977structure} himself observed that the question $Q(d+1)$ admits an affirmative answer. Subsequently, in 1988, Rao \cite{rao1988bass} established the question $Q(d)$ affirmatively. In addition, he proved in \cite{rao1991completing} that if $d = 3$, then $Q(2)$ has a positive answer. Before his work, in 1986, Roitman \cite{roitman1986stably} affirmed the question $Q(n)$ under an additional assumption that $R$ has a positive characteristic and $n\geq d/2+1$.  
\
 
The work of Gubeladze in the articles \cite{dzh1988gubeladze}, \cite{gubeladze1992elementary}, \cite{gubeladze1993elementary} led to the emergence of a parallel line of development. After decades of sustained effort that began with the said articles, Gubeladze finally in 2018 completed his project in \cite{gubeladze2018unimodular}  by proving that a unimodular row of length $n$ over the monoid algebra $R[M]$ can be completed to an elementary matrix if  $M$ is a commutative, cancellative monoid and $n\geq \max{(d + 2, 3)}$. This result motivated us to study the analogous question for $Q(n)$ over graded $R$-subalgebras of the polynomial algebra $R[t_1, \ldots, t_n]$. We proved in \cite{garggupta2025graded} that if $A$ is a graded $R$-subalgebra of $R[t, \frac{1}{t}]$, then any unimodular row over $A$ of length $n\geq \max{(d + 2, 3)}$ is the first row of an elementary matrix. We present the main result of this article below in a weaker form for the sake of clarity and better appreciation.
\begin{customthm}{A}{\rm (see Theorem \ref{mainpk2})} \label{MT}
Let $R$ be a Noetherian local ring of dimension $d$ $(\geq 2)$ such that $(d!) R = R$ and $A$ be a graded $R$-subalgebra of $R[t]$. Then any unimodular row of length $d + 1$ can be realized as the first row of some invertible matrix.     
\end{customthm}
The above theorem generalizes Rao's result \cite{rao1988bass}. We emphasize that the hypothesis on $A$ does not require it to be Noetherian. It subsumes well-known classes of algebras, such as Rees algebras and Symbolic Rees algebras over $R$, among others. The Swan–Weibel map ensures the validity of Theorem \ref{MT}, contingent upon an affirmative answer to the question $Q(d)$  in its stipulated setting. Consequently, Theorem \ref{MT} provides a necessary condition for an affirmative answer to $Q(d)$. 

Before proceeding to explain the methodology, we briefly digress. The set of unimodular rows of length $n$ is denoted by $\Um_{n}(R)$. The group	$\E_{n}(R)$ generated by elementary matrices in $\GL_{n}(R)$ acts naturally on $\Um_{n}(R)$ via right multiplication. The orbit space is denoted by $\frac{\Um_{n}(R)}{\E_{n}(R)}$. If $\dim R = d = 2$, then Suslin and Vaserstein \cite{vasersteinSuslin1976serre} established a bijection between the orbit space $\frac{\Um_{3}(R)}{\E_{3}(R)}$ and the Elementary Symplectic Witt group $\W_{\E}(R)$, thereby endowing $\frac{\Um_{3}(R)}{\E_{3}(R)}$ with the structure of an abelian group. Later,  van der Kallen \cite{van1983group} extended Vaserstein's group structure to $\frac{\Um_{d +1}(R)}{\E_{d +1}(R)}$ if $d\geq 3$. A modern interpretation attributed to Fasel \cite{fasel2010some} states that if $A$ is a smooth algebra of dimension $d \geq 3$ over a perfect field $k$ of characteristic not equal to $2$, then the orbit space $\Um_{d+1}(A)/\E_{d+1}(A)$ is isomorphic to the cohomology group, $H^d(A, G^{d+1})$. The study of such orbit spaces and their algebraic properties has long been a rich and intriguing area of research. If $X$ is a smooth real affine variety of dimension $d \geq 2$ and  $\RR(X)$ denotes the ring of regular functions on the set of its real closed points $X(\RR)$, then, the structure of the orbit space $\frac{\Um_{d +1}(\RR(X))}{\E_{d +1}(\RR(X))}$ has recently been described in \cite{das2018orbit}.
Our next result, which pertains to the structure of orbit spaces over graded subalgebras of $R[t]$, broadly generalizes a result of Rao \cite[Corollary 2.3]{rao1988bass}. For clarity, we state a weaker version here.
The term "good" refers to a technical condition which we introduce in Definition \ref{defpk} of the article.

\begin{customthm}{B}{\rm (see Theorem \ref{maint})}\label{MT1}
    Let $R$ be a local ring of dimension $d$ $(\geq 2)$ such that $2kR=R$ for some $k \in \NN$ and $A$ be a graded $R$-subalgebra of $R[t]$. Then $\frac{\Um_{d+1}(A)}{\E_{d+1}(A)}$ is good and $k$-divisible. 
 \end{customthm}

There has been significant progress in understanding the \emph{completion problem} for unimodular rows over affine algebras, a line of work pioneered by Suslin. Suppose $B$ is an affine algebra of dimension $d$ over a field $k$. Suslin \cite{suslin1984cancellation} proved that when $k$ is an infinite perfect $C_1$-field and $(d!)k = k$, then any unimodular row of length $d + 1$ over $B$ can be completed to an invertible matrix. Later, Fasel et al. \cite{fasel2012stably} proved that the same holds for unimodular rows of length $d$ over $B$ if $k$ is algebraically closed, $B$ is normal and $((d- 1)! ) k = k$. Mohan Kumar's work  \cite{kumar1985stably} shows that one cannot, in general, expect that unimodular rows of length $d - 1$ are completable. It remains unknown if the result of  Fasel et al. \cite{fasel2012stably} holds when $B$ is not normal. 

Our method of proof is inspired by the strategies employed by Suslin \cite{suslin1984cancellation} and  Rao \cite{rao1988bass}, albeit in a completely different setup. If $R$ and $A$ satisfy the hypothesis of Theorem \ref{MT}, then it follows from Theorem \ref{MT1} that a  unimodular row $(v_0, v_1, \ldots, v_d) \in \Um_{d + 1}(A)$ can be transformed, via elementary operations, into a unimodular row of the form $(w_0, w_1, \ldots, w_d^{d!})$. Since the latter is always the first row of an invertible matrix \cite{suslin1977stably}, Theorem \ref{MT} becomes an immediate consequence of Theorem \ref{MT1}. 
Therefore, we concentrate exclusively on the proof of Theorem \ref{MT1}. 
We shall henceforth assume that both $R$ and $A$ are as specified in Theorem \ref{MT1}. Furthermore, we may assume, without loss of generality, that $A$ is finitely generated due to Lemmas \ref{retract}, \ref{pkhom}, \ref{47}.

The proof proceeds by induction on the dimension $d$ of $R$.  Given $\v \in \Um_{d + 1}(A)$, one familiar with the aforementioned results of Suslin, Rao and Fasel et al. will naturally consider the construction of a map (homomorphism) $\frac{\Um_{d }(A/ xA)}{\E_d(A/xA)} \rightarrow \frac{\Um_{d +1 }(A)}{\E_{d+1}(A)}$ given by $[(\overline{a_1}, \ldots, \overline{a_d})] \mapsto [(a_1, \ldots, a_d, x)]$ for some nonzero divisor $x$ of $R$, such that the image of this map contains the class $[\v] \in \frac{\Um_{d +1 }(A)}{\E_{d+1}(A)}$. The existence of such a map enables the inductive step in the context considered by the said authors, since the property of being an affine algebra (or a polynomial algebra) is preserved under passage to a quotient by a nonzero divisor of $R$. However, this approach does not apply in our setting since $A/xA$ need not be a subalgebra of a polynomial algebra. To circumvent this issue, we construct in Lemma \ref{imp} a variant of the above map in the relative setting, enabling the inductive step. 
The central tool in our construction is Lemma \ref{43}, which is motivated by the Artin-Rees lemma. 
Therefore, as a first step, we proceed to prove Theorem \ref{MT1} in the relative setting. More specifically, we establish, via induction on $d$, that the group $\frac{\Um_{d+1}(A,IA)}{\E_{d+1}(A, IA)}$ is both good and $k$-divisible, assuming that $I$ is an ideal contained in the maximal ideal of $R$.

Suppose $\dim R = d = 2$. It follows from Proposition \ref{div} that the Witt group $\W_{\E}(A)$ is $2$-divisible. The Swan-Weibel homotopy, together with the analogous result of Rao for the polynomial algebra $R[t]$  plays a key role in the proof of Proposition \ref{div}. The relative version of Theorem \ref{MT1} is established in Proposition \ref{42} by using the bijection between $\W_{\E}(A)$ and $\frac{\Um_{3}(A)}{\E_{3}(A)}$, which follows from \cite{vasersteinSuslin1976serre}. The proof makes essential use of Rao’s results, \cite[Lemmas 1.3.1, 1.5.1]{rao1988bass}, along with Lemmas \ref{retract}, \ref{pkhom}, \ref{47}.

Now we assume that $\dim R = d \geq 3$. We prove Proposition \ref{RSL} by induction on $d$, which implies that the group $\frac{\Um_{d+1}(A,IA)}{\E_{d+1}(A, IA)}$ is both good and $k$-divisible if $I$ is a proper ideal of $R$ satisfying the additional condition $It \subset A$. This condition is necessary for the inductive step, as it ensures that the hypotheses of Lemma \ref{imp} are satisfied throughout the argument.
The proof relies on two main ideas. First, we invoke Remark \ref{mainreduct}, which allows us to assume further that $I$ is a principal ideal generated by a nonzero divisor and that $R$ is reduced. Second, we develop an analogue of Roitman's degree reduction technique for algebras that is valid in this setting,  see Lemmas \ref{amonic}, \ref{RT0}, and \ref{RT}. To remove the condition $It \subset A$ from the hypothesis of Proposition \ref{RSL}, we restrict our attention to the case where $R$ is a domain. By replacing $A$ with a suitable subintegral extension in $R[t]$ and applying Lemmas  \ref{subrel}, \ref{mainprin}, we reduce to a setting in which the inclusion $It \subset A$ holds. This reduction enables the proof of Lemma \ref{finite}, which establishes the relative version of Theorem \ref{MT1} when $R$ is a domain. The complete relative version then follows from Proposition \ref{57}, which is proved by induction on the number of minimal prime ideals of $R$. Here, Lemmas \ref{IJ}, \ref{411} play an essential role. This concludes the first step of the proof.

The results developed so far enable us to prove that $ \frac{\Um_{d + 1}(R[t], (t))}{\E_{d + 1}(R[t], (t))}$ is good and $k$-divisible, as stated in Proposition \ref{mgkd}. Let  $\m$ denote the unique maximal ideal of $R$ and set $B = A/ \m A$.  As an application of Proposition \ref{mgkd}, we prove Corollary \ref{mrgrg} which implies that the orbit space $\frac{\Um_{d + 1}(B)}{\E_{d + 1}(B)}$ is also good and $k$-divisible. In \ref{maineq1} of Theorem \ref{maint}, we construct a $3$-term exact sequence of abelian groups where $\frac{\Um_{d + 1}(A)}{\E_{d + 1}(A)}$ is sitting between $\frac{\Um_{d + 1}(A, \m A)}{\E_{d + 1}(A, \m A)}$ and $\frac{\Um_{d + 1}(B)}{\E_{d + 1}(B)}$. The later two orbit spaces are already shown to be good and $k$-divisible.  From this, we deduce  that $\frac{\Um_{d + 1}(A)}{\E_{d + 1}(A)}$ is also good and $k$-divisible, proving Theorem \ref{maint} which is a more general version of Theorem \ref{MT1}.

We now briefly outline the structure of the article. In Section 2, we provide basic definitions, fix notation, and recall necessary results that are used throughout this article. In Section 3, we present several preparatory results and develop the essential machinery needed to prove the main theorems. Section 4 is devoted to the special case where dimension of the ring $R$ is $2$. The proofs of the main theorems appear in Section 5. 

Our results may suggest several promising avenues for further research, especially regarding the efficient generation of ideals of graded subalgebras of polynomial algebras, the obstruction theory of projective modules over such algebras and the \emph{completion problem} for unimodular rows in this context.  Just as Rao’s work led to the development of the Euler class group of Noetherian rings \cite{bhatwadekar2000euler} and subsequently of polynomial rings \cite{das2003euler}, our framework suggests the possibility of developing a corresponding theory for graded $R$-subalgebras of $R[t]$. The algebras considered in this work are generally non-smooth, except in the case where they are polynomial algebras. Therefore, Theorem \ref{MT} allows us to construct families of non-smooth graded algebras over which the aforementioned theorem of Fasel et al. remains valid, see Example \ref{mainex}. 

The absence of a monic inversion principle in this context presents a fundamental obstruction to extending Theorem \ref{MT} to graded $R$-subalgebras of polynomial algebras over $R$ in several variables. In this regard, we conclude the article with a collection of questions that we believe will stimulate further research interest in this area. The techniques we employ are primarily based on elementary ideas and offer a conceptual clarity that we hope will make the results accessible to a broad mathematical audience.

{\bf Throughout this article, $R$ denotes a commutative Noetherian ring with $1 \neq 0$ unless anything is specifically stated otherwise. }

\section{Preliminaries}
We begin by reminding the reader of a few elementary definitions. A row vector $\v = (v_1, \ldots, v_n)$ $\in R^n$ is said to be unimodular if there exists a row vector $\w = (w_1, \ldots, w_n) \in R^n$ such that $\v \w^t = \sum _{i =1}^{n}v_iw_i = 1$.  The set of unimodular rows in $R^n$ is denoted by $\Um_n(R)$. Let $e_i$ denote the $i$-th row vector of the identity matrix $I_n$ of size $n$. 
The set of all unimodular rows of length $n$ which are congruent to $e_1$ modulo an ideal $I$ is denoted by $\Um_n(R, I)$. Vaserstein proved that any $\v \in \Um_n(R, I)$ corresponds to a $\w \in \Um_n(R, I)$ such that $\v \w^t= 1$  \cite[Lemma 2]{vaserstein1969stabilization}. Of course $\Um_n(R) = \Um_n(R, R)$.

Recall that an elementary matrix $e_{ij}(\lambda)$ is a square matrix of size $n$ which has $1$'s on the main diagonal, $\lambda$ at the $(i, j)$-th place and $0$'s elsewhere. The group $\E_n(R)$ is generated by elementary matrices $e_{ij}(\lambda)$ for $ \lambda \in R, i \neq j, 1 \leq i, j \leq n$. If $I$ is an ideal of $R$, then $\E_n(I)$ is defined as the subgroup of $\E_n(R)$ generated by the elementary matrices $e_{ij}(\lambda)$ for $i \neq j$, $1 \leq i, j \leq n$ and $\lambda \in I$. The normal closure of $\E_n(I)$ in $\E_n(R)$ is denoted by $\E_n(R, I)$. We define $\GL_n(R, I) = \{\alpha \in \GL_n(R) : \alpha \equiv I_n \vpmod I\}$ and $\SL_n(R, I) = \GL_n(R, I) \cap  \SL_n(R)$. Clearly, $\E_n(R, I) \subset \SL_n(R, I)$.

Any subgroup $G$ of $\SL_n(R, I)$ acts on $\Um_n(R, I)$. The orbit of $\v \in \Um_n(R, I)$ in the orbit space $\frac{\Um_n(R, I)}{G}$ is denoted by $[\v]$. For $\v, \w \in \Um_n(R, I)$, we  say that $\v$ is obtained by  applying $\E_n(R, I)$-operations on $\w$ and write $\v \sim_{\E_n(R, I)} \w$ if $\v = \w \alpha$ for some $\alpha \in \E_n(R, I)$ . 
 If $f : R \rightarrow S$ is a ring homomorphism and $I$, $J$ are ideals of $R$, $S$ respectively such that $f(I) \subset J$, then $f$ induces naturally a set map  $f_* : \frac{\Um_n(R, I)}{\E_n(R, I)} \rightarrow \frac{\Um_n(S, J)}{\E_n(S, J)}$ given by $f_* [(a_i)_{1 \times n}] = [(f(a_i))_{1 \times n}]$. For all other unexplained notation and terminology, we refer the reader to \cite{lam2006serre}.

\subsection{Maximal Spectrum}\label{magr}
The set of maximal ideals of the ring $R$ is denoted by $\maxSpec R$. It has a subspace topology endowed by the Zariski topology on $\Spec R$. A closed subset of $\maxSpec R$ is given by $V_M(I) = \{\m : \m \in \maxSpec R, I \subset \m\}$ for some ideal $I$ of $R$. Set $D_M(I) = \maxSpec R \setminus V_M(I)$. The Jacobson radical of $R$  is denoted by $\rad(R)$.

\subsection{Excision algebras}\label{preexa}
If $I$ is a proper ideal of a ring $R$, we consider $R \oplus I$ as a ring under coordinate-wise addition and multiplication given by $(r,  i)(s,  j) = (rs,  rj+si+ij)$ for $r, s \in R$ and 
$i,j \in I$. The ring homomorphism $\i : R \rightarrow R \oplus I$, $\i(r) = (r, 0), r \in R$ endows $R \oplus I$ an $R$-algebra structure, called excision algebra. There is a retraction map $\e : R \oplus I \mapsto R$, $\e(r, i) = r$ for $(r, i) \in R \oplus I$ of which $\i$ is a section. One can show that $R \oplus I$ is isomorphic to the fiber product $R \times_{R/ I} R$ under the map $(r, i) \mapsto (r, r + i)$ for $(r, i) \in R \oplus I$. The extension $\i : R \rightarrow R \oplus I$ is an integral extension \cite[Proposition 3.1, Corollary 3.2]{keshari2009cancellation}. 
If $R$ is a local ring of dimension $d$, then so is $R \oplus I$ \cite[Lemma 4.3]{gupta2014nice}.

\subsection{Excision rings, Excision theorem} \label{preexr}
If $I$ is a proper ideal of a ring $R$, one constructs the excision ring $\ZZ \oplus I$ with component-wise addition and multiplication defined by 
$(n,  i)(m,  j) = (nm,  nj+mi+ij)$ for $m,n \in \ZZ, i,j \in I$.
There is a ring homomorphism $\pi: \ZZ \oplus I \longrightarrow R$,  $\pi (m , i) = m+i$ for $(m, i) \in \ZZ \oplus I$. Let $\v = (1 +i_1, i_2, \ldots, i_n) \in \Um_n(R, I)$ where $i_j$'s are in $I$. Define $\tilde{\v} = (\tilde{1} + \tilde{i_1}, \tilde{i_2}, \ldots, \tilde{i_n}) \in \Um_n(\ZZ \oplus I, 0 \oplus I)$ for $\tilde{1} = (1, 0), \tilde{i_j} = (0, i_j)$. Clearly $\pi$ sends $\tilde{\v}$ to $\v$.

 The excision theorem due to van der Kallen states that the natural maps  $ \frac{\Um_n(\ZZ \oplus I, 0 \oplus I)}{\E_n(\ZZ \oplus I, 0 \oplus I)} \longrightarrow \frac{\Um_n(R, I)}{\E_n(R, I)}$  defined by $[(a_i)_{1 \times n}] \mapsto [(\pi(a_i))_{1 \times n}]$ and  $ \frac{\Um_n(\ZZ \oplus I, 0 \oplus I)}{\E_n(\ZZ \oplus I, 0 \oplus I)} \longrightarrow \frac{\Um_n(\ZZ \oplus I)}{\E_n(\ZZ \oplus I)}$ defined by $ [(a_i)_{1 \times n}] \mapsto [(a_i)_{1 \times n}]$ are bijections whenever $n \geq 3$ \cite[Theorem 3.21]{van1983group}.

\subsection{Symplectic Witt Groups}\label{swg}
Let $R$ be a ring which is not necessarily Noetherian. 
If $\alpha \in M_{2r}(R)$, $\beta \in M_{2s}(R)$, $r, s \in \NN$, then one defines $ \alpha \perp \beta = \begin{pmatrix}
        \alpha & 0\\
        0 & \beta 
    \end{pmatrix}$ .
    Let     
    $\Psi_1= \begin{pmatrix}
        0 & 1 \\-1 & 0
    \end{pmatrix}$, $\Psi_{r+1}= \Psi_{r} \perp \Psi_1$ for $r \geq 1$.  
Define $\SL(R) = \varinjlim \SL_n(R)$  and $\E(R)  = \varinjlim \E_n(R)$. Suppose $G$ denotes either $\E(R)$ or $\SL(R)$.

We recall symplectic Witt groups as defined by Vaserstein and Suslin in \cite[\S3]{vasersteinSuslin1976serre}.   
Let $S$ be the collection of all alternating matrices of Pfaffian $1$.  For $\alpha \in \M_{2r}(R) \cap S$ and $\beta \in \M_{2s}(R) \cap S$,  $r,s\in \NN$, we define $\alpha \sim \beta$, if there exists $\varepsilon \in G \cap \M_{2(r+s+l)}(R)$  such that $\varepsilon^t(\alpha \perp \Psi_{s+l})\varepsilon=\beta \perp \Psi_{r+l}$ for some $l\in \NN$.  The relation $\sim$ is an equivalence relation and the set of all equivalence classes $S/ \sim$ equipped with the $\perp$ operation forms an abelian group. This group is denoted by $\W_{\SL}(R)$ when $G = \SL(R)$ and is called the symplectic Witt group. Likewise, when $G = \E(R)$, it is denoted by $\W_{\E}(R)$ and is called the elementary symplectic Witt group. Both associations $R \rightarrow \W_{\SL}(R)$,  $R \rightarrow \W_{\E}(R)$ are functorial.

 In the same article, authors defined a map $\frac{\Um_3(R)}{G \cap \SL_3(R)} \rightarrow \W_{\E}(R)$ given by 
\[ [(a_1,a_2,a_3)]  \mapsto V(a,b) = \begin{pmatrix}
        0 & -b_1 & -b_2 & -b_3\\
       b_1 & 0 & -a_3 & a_2 \\
       b_2 & a_3 & 0 & -a_1 \\
       b_3 & -a_2 & a_1 & 0 
       \end{pmatrix}    \]
where $(a_1, a_2, a_3) \in \Um_3(R)$ and $(b_1, b_2, b_3) \in \Um_3(R)$ are chosen such that $a_1b_1 + a_2b_2 + a_3b_3 = 1$. 

It follows from \cite[Theorem 5.2]{vasersteinSuslin1976serre}, that the above map is bijective if $R$ satisfies the condition that $\Um_n(R) = e_1 \E_n(R)$ for all $n \geq 4$. Moreover, it is implicit in the proof of  \cite[Corollary 7.4]{vasersteinSuslin1976serre} that $\frac{\Um_3(R)}{\E(R) \cap \SL_3(R)} = \frac{\Um_3(R)}{\E_3(R)}$ when the  aforementioned condition and the condition $\SL_4(R) \cap \E(R) = \E_4(R)$ both are satisfied. There are two well-known cases where these conditions are readily satisfied.

The first case concerns rings of type $R$ for which $\maxSpec R$ is a union of finitely many Noetherian subspaces of dimension at most $2$. The second case consists of polynomial rings of type $R_0[X_1, \ldots, X_n]$ where $R_0$ is a ring of dimension $2$, see \cite[Theorems 2.6,  7.8]{suslin1977structure}. In both cases we immediately check that $\frac{\Um_3(R)}{\SL_3(R)} \cong \W_{\SL}(R)$ and $\frac{\Um_3(R)}{\E_3(R)} \cong \W_{\E}(R)$ have the structure of an abelian group. The group operation is described in (\ref{gop}), \S \ref{pregst} below.

In \cite[Proposition 2.2]{banerjee2023stability},  Banerjee proved the following lemma in the special case where $R$ is a domain. For our purpose, we need a more general version.

\begin{lemma}\label{maindim}
Let $R$ be a ring and $A =\oplus_{i\geq 0} A_i$ be a finitely generated graded $R$-algebra of dimension $d$ such that $A_0 = R$. If $s \in \rad(R) \setminus \{0\}$, then $\dim(A_s) < d$.
\end{lemma}

\begin{proof}
Suppose, on the contrary, we have $\dim(A_s) = d$. The ring $A_s$ is a Noetherian positively graded ring, so it admits a homogeneous maximal ideal of height $d$, see \cite[Exercise 1.5.25]{bruns1998cohen}. It follows that there is a homogeneous prime ideal $\n = \oplus_{i\geq 0 } \n_i$ of $A$ such that  $s \not \in \n$ and $\height (\n) = d$. Clearly $\n \subset \m \oplus (\oplus_{i\geq 1 } A_i)$ where $\m$ is a maximal ideal of $R$ containing $\n_0$. As $\height \n = \dim A$, the prime ideal $\n$ is a maximal ideal of $A$, so $\n = \m \oplus (\oplus_{i\geq 1 } A_i)$. This is a contradiction since $s \in \m$ but $s \not\in \n$. Hence, the proof follows. 
\end{proof}

\subsection{A group structure on the orbit space}\label{pregst}
Let $R$ be a ring which is not necessarily Noetherian, and $I$ be an ideal of $R$. Suppose $D_M(I)$ is a union of finitely many  Noetherian subspaces of dimension at most $d \geq 2$. Extending the result of Suslin and Vaserstein \cite[Theorem 5.2]{vasersteinSuslin1976serre},  van der Kallen defined an abelian group structure on $ \frac{\Um_{d + 1}(R, I)}{\E_{d + 1}(R, I)}, d \geq 2$, see  \cite[Theorem 3.6]{van1983group}. We describe it here. 

Suppose $[\v], [\w] \in \frac{\Um_{d + 1}(R, I)}{\E_{d + 1}(R, I)}$ for $\v, \w  \in \Um_{d + 1}(R, I)$.
We may assume that $ \v = (a, a_1, a_2, \ldots, a_{d })$ and $\w = (b, a_1, a_2, \ldots, a_{d })$, see \cite[Lemma 3.4]{van1983group}. 
Let $ap \equiv 1 \vpmod{(a_1, \ldots, a_d)}$ for $p \in R$. 
The orbit space $\frac{\Um_{d + 1}(R, I)}{\E_{d + 1}(R, I)}$ is an abelian group under the operation given by 
\begin{equation}\label{gop}
[\w] [\v] = [(a(b + p) -1, a_1(b + p), a_2, \ldots, a_{d})].
\end{equation}
Here $[e_1]$ is the identity element.

 Later in \cite[Theorem 4.1]{van1989module}, van der Kallen strengthened his result by proving that $\frac{\Um_{n + 1}(R, I)}{\E_{n + 1}(R, I)}$ has a group structure under the same operation if $n \geq \max \{1 + \frac{d}{2}, 2\}$. 
It is shown in \cite[Theorem 3.16]{van1983group}, \cite[Lemma 3.5 (v)]{van1989module} that the group operation (\ref{gop}) takes particularly simple form when one of the coordinates of either $\v$ or $\w$ is a perfect square. Let $a = a'^2$, then
\begin{equation}\label{gop1}
[(b, a_1, a_2, \ldots, a_{d })] [(a'^2, a_1, a_2, \ldots, a_{d })] = [(b a'^2, a_1, a_2, \ldots, a_{d})].
\end{equation}

It is worth pointing out here that if $f : R \rightarrow S$ is a ring homomorphism and $I$, $J$ are ideals of $R$, $S$ respectively such that $f(I) \subset J$, then the induced map $ f_*: \frac{\Um_n(R, I)}{\E_n(R, I)} \rightarrow \frac{\Um_n(S, J)}{\E_n(S, J)}$ is a group homomorphism if both orbit spaces have group structures as in (\ref{gop}). 

The following examples elucidate the group structure of specific orbit spaces, which are essential for the proof of the main results.

\begin{example}\label{prex}
Suppose $R$ is a ring of dimension $d (\geq 2)$ and $A$ is a finitely generated graded $R$-subalgebra of the polynomial ring $R[t]$. We deduce from the dimension inequality \cite[Theorem 15.5]{matsumura1989commutative} that $\dim A \leq d + 1$. Let $I = (x_1, \ldots, x_r) \subset \rad(R)$ be an ideal. We observe the following. 
\begin{enumerate}
\item
It is clearly seen that $D_M(IA) = \cup_{i = 1}^r D_M(x_iA)$ and $D_M(x_iA) \subset \Spec A_{x_i}$ which has dimension at most $d$ by Lemma \ref{maindim}. Therefore, $D_M(IA)$ is a union of finitely many Noetherian subspaces of dimension at most $d$. The orbit space $\frac{\Um_{d + 1}(A, IA)}{\E_{d +1}(A, IA)}$ is an abelian group under the operation as in (\ref{gop}), \S \ref{pregst}.

\item 
Now we further assume that $\height \rad(R) \geq 1$.  We find $s \in \rad(R)$ such that $\height (s) = 1$. Any minimal prime ideal of $A$ is of the form $\p R[t] \cap A$ where $\p$ is a minimal prime ideal of $R$, so it cannot contain $s$. Therefore, we have $\height (sA) = 1$ and consequently, $\dim (A/ sA) \leq d$. We also have $\dim A_s \leq d$ by Lemma \ref{maindim}. It follows that $\maxSpec A = V_M(sA) \cup D_M(sA)$ is a union of two Noetherian subspaces of dimension at most $d$ and consequently so is $D_M(J)$ for any ideal $J$ of $A$. Hence, $\frac{\Um_{d + 1}(A, J)}{\E_{d +1}(A, J)}$ has an abelian group structure as in (\ref{gop}), \S \ref{pregst} for any ideal $J$ of $A$.
\end{enumerate}

\end{example}

The lemma below is well-known. In the absence of a suitable reference, we include a proof for completeness.

\begin{lemma}\label{retract}
Let $R$ be a ring which is not necessarily Noetherian. Assume that  $\maxSpec R$ is a union of finitely many Noetherian subspaces of dimension at most $d$. Let $q : R \rightarrow S$ be a retraction map and $I = \ker q$. Then the following sequence is exact. 
\[
0 \rightarrow \frac{\Um_{d+1}(R, I)}{\E_{d+1}(R, I)}\xrightarrow{\alpha} \frac{\Um_{d+1}(R)}{\E_{d+1}(R)} \xrightarrow{q_*} \frac{\Um_{d+1}(S)}{\E_{d+1}(S) }\rightarrow 0
\]
Here $\alpha$ is the natural map $[\v]\mapsto [\v]$ for  $ \v \in \Um_{d+1}(R, I)$ and $q_*$ is the map induced by $q$. 
\end{lemma}

\begin{proof}
Let $i : S \rightarrow R$ be a section of $q$, i.e. $q \circ i = id_S$. We may view $S$ as a subring of $R$ via the injective map $i$. Then $q : R \rightarrow S$ is a surjection such that its restriction to $S$ is the identity map on $S$. By \S \ref{pregst}, all orbit spaces involved have abelian group structures and the intermediate maps $\alpha, q_*$ are group homomorphisms.  Clearly, $q_*$ is a split surjection with the section $i_*$ induced by $i$. 

Suppose $[\w] \in \ker q_*$ for $\w \in \Um_{d +1}(R)$. If $\overline{\w}$ denotes the image under $q$, then there is a $\tau \in \E_{d + 1}(S) \subset \E_{d + 1}(R)$ such that $\overline{\w} \tau = e_1$. Clearly the image of $\w \tau$ under $q$ is $e_1$, so $\w \tau \in \Um_{d + 1}(R, I)$. It follows that $[\w] = [\w \tau] \in \Imj \alpha$. Therefore, $ \ker q_* \subset \Imj \alpha$ and consequently $ \ker q_*  =  \Imj \alpha$ as the reverse inclusion is clear. 

 Let $[\v] \in \ker \alpha$ for $\v \in \Um_{d + 1}(R, I)$. Then there is a $\tau \in \E_{d + 1}(R)$ such that $\v \tau = e_1$. 
If $\overline{\tau} \in \E_{d + 1}(S)$ denotes the image of $\tau$ under $q$, we have $e_{1} \overline{\tau} = e_{1}$. Let $\eta = \tau \overline{\tau}^{-1} $, then we have $\v \eta = e_1$ and $\eta \in \E_{d+1}(R)\cap \SL_{d+1}(R, I)= \E_{d+1}(R, I)$ by \cite[p. 204, Lemma 1.5]{lam2006serre}. It follows that $[\v] = [e_1]$ in $\frac{\Um_{d+1}(R, I)}{\E_{d+1}(R, I)}$. Therefore, $\ker \alpha$ is trivial and $\alpha$ is injective. The proof is complete here.   
\end{proof}

\begin{remark}\label{retract1}
If we weaken the hypothesis of Lemma \ref{retract} by saying that $q : R \rightarrow S$ is only a surjective homomorphism, then the argument in the proof shows that  
$\frac{\Um_{d+1}(R, I)}{\E_{d+1}(R, I)}\rightarrow \frac{\Um_{d+1}(R)}{\E_{d+1}(R)} \rightarrow \frac{\Um_{d+1}(S)}{\E_{d+1}(S)}$ is exact.
\end{remark}

\subsection{An exact sequence}\label{preexact}
   Let $R$ be a ring which is not necessarily Noetherian and $\maxSpec R$ be a union of finitely many Noetherian subspaces of dimension at most $d$. Let  $I = (x_1, \ldots, x_n)$ be a proper ideal of $R$. We observe that $(R \oplus I)_{(0, x_i)} = R_{x_i}$, therefore $D_M(0, x_i) = D_M(x_i)$. Moreover,
 \[
\maxSpec (R \oplus I) = V_M(0 \oplus I) \cup D_M(0, x_1)  \cup \ldots \cup D_M(0, x_n)
= \maxSpec R \cup D_M(x_1)  \cup \ldots \cup D_M(x_n).
\]   
Each of $\maxSpec R$, $D_M(x_1), \ldots, D_M(x_n)$ is a union of finitely many Noetherian subspaces of dimension at most $d$, therefore so is $\maxSpec (R \oplus I)$.

We observe that $\e : R \oplus I \rightarrow R$ is retraction  map with $\ker \e = 0 \oplus I$, so by Lemma \ref{retract}, we have the following short exact sequence
\[0 \rightarrow \frac{\Um_{d + 1}(R \oplus I, 0 \oplus I)}{\E_{d + 1}(R \oplus I, 0 \oplus I)} \xrightarrow{\alpha} \frac{\Um_{d + 1}(R \oplus I)}{\E_{d + 1}(R \oplus I)}  \xrightarrow{\e_*} \frac{\Um_{d + 1}(R)}{\E_{d + 1}(R)} \rightarrow 0.\]
By the double excision theorem  \cite[Lemma 3.1]{gupta2014nice}, there is an isomorphism $\mathfrak{f} : \frac{\Um_{d + 1}(R, I)}{\E_{d + 1}(R, I)} \rightarrow \frac{\Um_{d + 1}(R \oplus I, 0 \oplus I)}{\E_{d + 1}(R \oplus I, 0 \oplus I)}$. If we choose $\j = \alpha \circ \mathfrak{f}$, we get the exact sequence below. 
\[0 \rightarrow \frac{\Um_{d + 1}(R,  I)}{\E_{d + 1}(R, I)} \xrightarrow{\j} \frac{\Um_{d + 1}(R \oplus I)}{\E_{d + 1}(R \oplus I)}  \xrightarrow{\e_*} \frac{\Um_{d + 1}(R)}{\E_{d + 1}(R)} \rightarrow 0.\]
One observes that $[(1 + i_0, i_1, \ldots, i_d)] \xrightarrow{\j} [(1, i_0), (0, i_1), \ldots, (0, i_d))]$ for $(1 + i_0, i_1, \ldots, i_d) \in \Um_{d + 1}(R, I)$.

The following result is proved in \cite[Theorem 2.4]{gupta2016existence}.
\begin{theorem}\label{AG}

Let $S$ be a multiplicatively closed subset of a ring $A$ and $R= S^{-1}A$ be a ring of dimension at most $d$.
 Let $(a_0, a_1,\ldots, a_n)\in \Um_{n+1}(R)$, $n \geq d +1$. 
 Then for any $s \in S$ there exist $c_1, c_2,\ldots, c_{d+1}\in sA$ such that $(a_1+c_1a_0, a_2+c_2a_0,\ldots, a_{d+1}+c_{d+1}a_0, a_{d+2},\ldots, a_n) \in \Um_{n}(R)$.
\end{theorem}

The following lemma is folklore. We outline a proof for completeness.

\begin{lemma}\label{26}\label{nil}
 Let $R$ be a ring and $I$ be an ideal of $R$. Suppose  $\v = (v_1, \ldots, v_n) \in \Um_{n}(R, I)$ and $x \in \Nil(R) \cap I$. Then $(v_1, \ldots,v_{i - 1}, v_i +x, v_{i + 1}, \ldots, v_n)\sim_{\E_n(R,I)} \v$. In particular, if $v_1$ is a unipotent element of $R$ then $\v \sim_{\E_{n}(R,I)} e_1$.
\end{lemma}

\begin{proof}
Let $\tilde{\v} \in \Um_n(\ZZ \oplus I, 0 \oplus I)$ be a lift of $\v$ whose image under $\pi$ is $\v$, see \S \ref{preexr}. We observe that $\tilde{x} = (0, x)$ is a nilpotent element of $\ZZ \oplus I$. Let $\mathbf{x} \in (\ZZ \oplus I)^n$ be a row vector whose only nonzero coordinate is $\tilde{x}$ at the $i$-th position. 
Then $\pi$ sends $\tilde{\v} + \mathbf{x}$ to $(v_1, \ldots,v_{i - 1}, v_i +x, v_{i + 1}, \ldots, v_n)$. By \cite[Lemma 1.4.2]{rao1988bass}, we have  $\tilde{\v} + \mathbf{x} \sim_{\E_n(\ZZ \oplus I)} \tilde{\v}$. Therefore, the first assertion follows from the excision theorem. To prove the second assertion, let $v_1=1+n$ where $n \in I$ is a nilpotent element. By the first assertion, we have $\v \sim_{\E_{n}(R,I)} (1, v_2, \ldots,v_n )$. Since $(1, v_2, \ldots,v_n ) \sim_{\E_{n}(R,I)} e_1$, the second assertion follows. 
\end{proof}

The next lemma is proved in \cite[Corollary 2.17]{chakraborty2024relative}. When $I = R$, it is due to Roitman \cite[Lemma 1]{roitman1986stably}.

\begin{lemma} \label{22}
Let $I$ be an ideal of a ring $R$ and $\v=(v_0,\ldots,v_n) \in \Um_{n+1}(R,I), n \geq 2$. Let  $t \in I$ be a unit modulo the ideal $(v_0,\ldots,v_{n-2})$. Then $$(v_0,\ldots,v_n) \sim_{\E_{n+1}(R,I)} (v_{0},\ldots,v_{n-1},t^2v_{n}).$$
\end{lemma}

\section{Some preparatory results}

This section aims to prove results that play a key role in the subsequent development. The following lemma facilitates various reductions and is frequently employed in later sections.
\begin{lemma}\label{IJ}
    Let $R$ be a ring and $I, J$ be ideals of $R$ such that $IJ \subset \Nil(R)$. Let overline denote the image modulo $J$ and $n\geq 3$. Then the natural map \[ \alpha: \frac{\Um_{n}(R, I)}{\E_{n}(R, I)}\rightarrow \frac{\Um_{n}(\overline{R},\overline{I})}{\E_{n}(\overline{R},\overline{I})} \] is a bijective map.
   \end{lemma}
\begin{proof} 
First we show that $\alpha$ is surjective.
We choose $\w= (\overline{1} + \overline{i_1}, \overline{i_2}, \ldots, \overline{i_n}) \in \Um_{n}(\overline{R},\overline{I})$ where $i_1, \ldots, i_n \in I$. Let  $K$ be the ideal generated by $1 + i_1, i_2,  \ldots, i_n$. Then, we have $K + I = R$. We also have $K + J = R$ since $\w \in \Um_{n}(\overline{R},\overline{I})$. 
As $IJ \subset \Nil(R)$, an easy argument yields $K = R$. It follows that $\v = (1 + i_1, i_2,  \ldots, i_n) \in \Um_{n}(R,I)$ and $\v$ is a preimage of $\w$, so $\alpha$ is surjective.

 To show that $\alpha$ is injective, we begin with $\v_1= (c_1,\ldots,c_n)$ and $\v_2= (d_1,\ldots,d_n) \in \Um_{n}(R,I)$ such that $\alpha([\v_1])=\alpha([\v_2])$. We have $\overline{\v_1} \sim_{E_n(\overline{R},\overline{I})} \overline{\v_2}$. Lifting the elementary operations, we obtain $\v_1 \sim_{\E_n(A, I)} (d_1+j_1,\ldots,d_n+j_n)$ where each $j_i$ is in $I \cap J$. As $IJ \subset \Nil(R)$, we have $I \cap J \subset \Nil(R)$, so each $j_i$ is a nilpotent element. 
By Lemma \ref{26}, we get  $(d_1+j_1,\ldots,d_n+j_n) \sim_{\E_n(A, I)} \v_2$. It follows that $\v_1 \sim_{\E_n(A, I)} \v_2$, so $[\v_1] = [\v_2]$ in $\frac{\Um_{n}(R, I)}{\E_{n}(R, I)}$. Therefore, $\alpha$ is injective. Hence, the proof follows. 
\end{proof}

If $I$ is an ideal of a ring  $R$, we define $\Gamma_{I} (R) = \cup_{n \in \NN} (0 : I^n)$. One can easily check that the image $\overline{I}$ of $I$ in the ring $R/\Gamma_I (R)$ contains a nonzero divisor if $\overline{I} \neq 0$, equivalently $I$ is not a nilpotent ideal. Moreover $I\Gamma_I(R) \subset \Nil(R)$. Therefore, Lemma \ref{IJ} implies the following.

\begin{corollary}\label{nzd}
  Let $I$ be an ideal of a ring $R$ and overline denote the image modulo $\Gamma_I(R)$.
  Then the natural map
$$\alpha: \frac{\Um_{d+1}(R, I)}{\E_{d+1}(R, I)} \longrightarrow \frac{\Um_{d+1}(\overline{R},\overline{I})}{\E_{d+1}(\overline{R},\overline{I})}$$ is a bijective map. 
\end{corollary}

The following lemma constitutes an essential component in the proof of Theorem \ref{maint}.

\begin{lemma}\label{411}
    Let $I, J$ be ideals of $R$ such that $IJ=0$. Then the natural map \[ \frac{\Um_{n}(R,I+J)}{\E_{n}(R,I+J)}\rightarrow \frac{\Um_{n}(\overline{R},\overline{I})}{\E_{n}(\overline{R},\overline{I})} \times \frac{\Um_{n}(\overline{R},\overline{J})}{\E_{n}(\overline{R},\overline{J})}, n \geq 3 \] is bijective. Here, overline denotes the image modulo $J$ in the first factor and the image modulo $I$ in the second factor.
\end{lemma}

\begin{proof}
We shall construct a map $\frac{\Um_n(R,I)}{\E_n(R,I)}\times\frac{\Um_n(R,J)}{\E_n(R,J)} \rightarrow \frac{\Um_n(R,I+J)}{\E_n(R,I+J)}$ which yields the following commutative diagram : 

$$
   \xymatrix{
\frac{\Um_n(R,I+J)}{\E_n(R,I+J)}  \ar[r]^{} & \frac{\Um_{n}(\overline{R},\overline{I})}{\E_{n}(\overline{R},\overline{I})}\times \frac{\Um_{n}(\overline{R},\overline{J})}{\E_{n}(\overline{R},\overline{J})}\\
  & \ar[u]_{\cong~ \text{by Lemma} \  \ref{IJ}}\ar[lu] \frac{\Um_{n}({R},{I})}{\E_{n}({R},{I})}\times \frac{\Um_{n}({R},{J})}{\E_{n}({R},{J})}.}
$$
where the horizontal map is  the given natural map, and the vertical  map is given by the pair of natural
bijections defined on different factors, see  Lemma \ref{IJ}. 

Let $(1+i_1, i_2, \ldots, i_n) \in \Um_n(R,I)$ and $(1+j_1, \ldots, j_n) \in \Um_n(R,J)$. We consider the ideal $T$ generated by $(1+i_1+j_1), (i_2+j_2), \ldots, (i_n+j_n)$.
Then $T+I= R$, $T+J=R$ and as $IJ= 0$ we have $T= R$. Consequently $(1+i_1+j_1,i_2+j_2,\ldots,i_n+j_n) \in \Um_n(R,I+J)$.
 
 We define the diagonal map 
\[\frac{\Um_n(R,I)}{\E_n(R,I)}\times\frac{\Um_n(R,J)}{\E_n(R,J)} \rightarrow \frac{\Um_n(R,I+J)}{\E_n(R,I+J)}\] as 
\[([(1+i_1,\ldots,i_n)],[(1+j_1,\ldots,j_n)])\mapsto [(1+i_1+j_1,i_2+j_2,\ldots,i_n+j_n)].\]

We prove that the map is well-defined. 
Let $[(1+i_1,\ldots,i_n)] = [(1+i_1',\ldots,i_n')] \in \frac{\Um_n(R,I)}{\E_n(R,I)}$ and $[(1+j_1,\ldots,j_n)] \in \frac{\Um_n(R,J)}{\E_n(R,J)}$. Then there is an $\alpha \in \E_n(R,I)$ such that $(1+i_1, \ldots, i_n)\alpha =(1+i_1', i_2',\ldots ,i_n')$. 
Now $\alpha$ can be written as $\alpha = I_n + \alpha'$ for some $\alpha' \in \M_n(I)$. As $IJ = 0$, we have $(j_1,j_2,\ldots,j_n) \alpha' = 0$, consequently $(j_1,j_2,\ldots,j_n) \alpha = (j_1,j_2,\ldots,j_n)$. Therefore, 
\[(1+i_1+j_1,i_2+j_2,\ldots,i_n+j_n) \alpha = (1+i_1'+j_1, i_2' + j_2, \ldots, i_n' + j_n).\]
 One checks similarly that the image is independent of the choice of representatives of the second factor. 
 
As the vertical map in the above diagram is injective, so is the diagonal map. If $$(1 + i_1 + j_1, i_2 + j_2, \ldots, i_n + j_n) \in \Um_n(R ,I+J),$$ then $(\bar{1} + \bar{i_1}, \bar{i_2}, \ldots, \bar{i_n}) \in \Um_n(\overline{R} , \overline{I})$ where overline denotes the image modulo $J$. The argument as in Lemma \ref{IJ} shows that $(1 + i_1, i_2, \ldots, i_n) \in \Um_n(R ,I)$. Similarly $(1 + j_1, j_2, \ldots,  j_n) \in \Um_n(R , J)$. Therefore, the diagonal map is surjective and consequently, bijective. The vertical map is already bijective. Hence, the horizontal natural map is also bijective. 
\end{proof}

\begin{remark}
The above lemma holds even for $IJ \subset \Nil R$.
\end{remark}

\begin{definition}
The leading coefficient of a polynomial $f \in R[t]$ is denoted by $l(f)$. The polynomial $f$  is called an $a$-monic polynomial for $a \in R \setminus \{0\}$ if $l(f) = a^n$ for some $n \in \NN$.
\end{definition} 

\begin{lemma}\label{amonic}
Let $R$ be a ring of dimension $d$ $(\geq 2)$ and $A=\oplus I_it^i$ be a graded $R$-subalgebra of $R[t]$ such that $I_1 \neq 0$. Suppose $\v=(v_0,\ldots,v_{d})\in \Um_{d+1}(A,aA)$ for some $a \in I_1 \cap \rad(R)$, $a \neq 0$. Then we have $\v \sim_{\E_{d+1}(A,aA)}(u_0,\ldots,u_{d})$ where  $u_0$ is an $a$-monic polynomial in $A$. 
\end{lemma}

\begin{proof} 
We begin with an observation that $\dim(R_a)\leq d-1$. Let $T$ be the multiplicatively closed subset of $A$ generated by all $a$-monic polynomials in $A$. Then $T^{-1} A = S^{-1} R_a[t]$ where $S$ is the multiplicatively closed subset of $R_a[t]$ generated by all monic polynomials in $R_a[t]$.  It follows that $\dim T^{-1} A = \dim S^{-1} R_a[t] = \dim R_a \leq d - 1$. 
 
The image of $\v$ under the localization  is in $\Um_{d+1}(T^{-1} A)$. By Theorem \ref{AG}, we can find $c_1,c_2,\ldots,c_d$ in $aA$ such that $(v_1+c_1v_0, v_2 + c_2v_0, \ldots, v_d + c_dv_0) \in \Um_{d}(T^{-1}A)$. It follows that the ideal generated by $v_1+c_1v_0, v_2 + c_2v_0, \ldots,v_d + c_dv_0$ contains an $a$-monic polynomial, say $h(t)$. In particular, it contains $a(at)^mh(t)$ for all  $m \in \NN$.

We  choose $m > \deg(v_0)$ such that $v_0 + a(at)^mh(t)$ is $a$-monic. Now the result follows once we observe that $\v \sim_{\E_{d+1}(A,aA)} (v_0, v_1+c_1v_0, \ldots, v_d + c_dv_0) \sim_{\E_{d+1}(A,aA)} (v_0 + a(at)^mh(t), v_1+c_1v_0, \ldots, v_d + c_dv_0)$. 
\end{proof}

The next two lemmas adapt the technique from Roitman's method to reduce the degrees of coordinates in unimodular rows over polynomial rings, see \cite[Theorem 5]{roitman1986stably}. This approach involves systematically lowering the degrees of all but one coordinate by applying elementary operations akin to Euclid's algorithm.

\begin{lemma}\label{RT0}
Let $A= \oplus_{i\geq 0} I_it^i$ be a graded $R$-subalgebra of $R[t]$ such that $I_1 \neq 0$ and
$$\v = (v_0,v_1,\ldots,v_d) \in \Um_{d + 1}(A, aA), d \geq 2 $$ for some $a \in I_1 \setminus \{0\}$. Assume that $v_0$ is $a$-monic, the ideal generated by the coefficients of $v_2, \ldots, v_d$ in $R_a$  is equal to the ideal $R_a$ and $m = \max \{\deg v_2, \ldots, \deg v_d\}$. Then by applying $\E_{d + 1}(A, aA)$-operations on $\v$, we may modify $v_1, \ldots, v_d$ such that $v_1$ is $a$-monic, $\deg v_1 \leq m$ and $\max \{\deg v_2, \ldots, \deg v_d\} \leq m -1$. 
\end{lemma}

\begin{proof}
By the given hypothesis, $v_0$ is monic in $R_a[t]$, the coordinates $v_2, \ldots, v_d$ have degrees at most $m$  and the  coefficients of $v_2, \ldots, v_d$ generate the ideal $R_a$. By \cite[Lemma 11.1]{vasersteinSuslin1976serre}, the ideal $(v_0, v_2, \ldots, v_d)R_a[t]$ contains a monic polynomial of degree at most $m$. Since $A_a = R_a[t]$, it follows that the ideal $(v_0, v_2, \ldots, v_d)A$ contains an $a$-monic polynomial say $u$ of degree at most $m$. Let  $l(u) = a^M$. 

We claim that by applying $\E_{d + 1}(A, aA)$-operations on $\v$, we may reduce the degree of $v_1$ to have $\deg v_1 < \deg u \leq m$.  If $\deg v_1 \geq \deg u$, we choose $M_0 \in \NN$ satisfying $2M_0 > M + \deg v_1 - \deg u$ and observe that
\begin{align}\label{deq1}
\v  &\sim_{\E_{d+1}(A,aA)}(v_0, a^{2M_0}v_{1}, \ldots, v_d) \ \text{by Lemma \ref{22} since $v_0 \equiv 1 \vpmod {aA}$}\\
(v_0, a^{2M_0}v_{1}, \ldots, v_d) & \sim_{\E_{d+1}(A,aA)} (v_0, a^{2M_0}v_1 -a^{2M_0 - M - \deg v_1 + \deg u} l(v_{1}) (at)^{\deg v_1 - \deg u}u(t), v_2 ,\ldots,v_d). \nonumber
\end{align}
For the second $ \sim_{\E_{d+1}(A,aA)}$ we use that $at \in A$.  
We notice that $$\deg [a^{2M_0}v_1 -a^{2M_0 - M - \deg v_1 + \deg u}l(v_{1}) (at)^{\deg v_1 - \deg u}u(t)] < \deg v_1.$$ 
Therefore, by applying suitable $\E_{d + 1}(A, aA)$-operations iteratively, we achieve our claim. We recall that $u \in (v_0, v_2, \ldots, v_d)A$. 
By applying further $\E_{d + 1}(A, aA)$-operations on $\v$, we add $au$ to $v_1$.  Thus, we may assume that  $\deg v_1 = \deg u \leq m$ and $v_1$ is $a$-monic. 

Finally, we apply $\E_{d + 1}(A, aA)$-operations on $\v$ similar to (\ref{deq1}) above  to reduce degrees of $v_2, \ldots, v_d$ such that $\max \{\deg v_2, \ldots, \deg v_d\} \leq m -1$. The proof is completed here. 
\end{proof}

\begin{lemma}\label{RT}
Let $R$ be a ring of dimension $d$ $(\geq 2)$  and  $A= \oplus_{i\geq 0} I_it^i$ be a graded $R$-subalgebra of $R[t]$. Assume that $I_1 \cap \rad(R)$ contains an element $a$ such that $\height (aR) = 1$.  Then for any $\v \in \Um_{d+1}(A, aA)$ we have
$\v \sim_{\E_{d+1}(A, aA)} (w_0,w_1,\ldots,w_d)$
such that the following hold : 
\begin{enumerate}
\item
$w_0$ is $a$-monic,
 \item
$w_1$ is a linear $a$-monic polynomial, 
 \item
$w_i\in aR$ for all $i= 2, \ldots, d$,
\item 
$\height(w_d) = 1$
\end{enumerate}
\end{lemma}

\begin{proof}
    Let $\v=(v_0,\ldots,v_d) \in \Um_{d+1}(A,aA)$. By Lemma \ref{amonic}, we may assume that $v_0$ is $a$-monic and $l(v_0)=a^n$. 
The proof naturally falls into two steps. 

\noindent
{\bf Step 1 :} {\it (Roitman's machine)
If $m = \max \{\deg v_2, \ldots, \deg v_d\} \geq 1$, then by applying $\E_{d+1}(A, aA)$-operation on $\v$, one can modify $\v$ such that $\max \{\deg v_2, \ldots, \deg v_d\} < m$, $\deg v_1 \leq m$ and $v_1$ is $a$-monic.}

Let $v_j = \sum_{i = 0}^{m} c_{ij} t^i$ for  $2 \leq j \leq d$ and $J$ be the ideal of $R$ generated  by the coefficients $c_{ij}$ for $0 \leq i \leq m$ and $2 \leq j \leq d$.      
 We observe that $(v_0,v_1)R[t] + JR[t]= R[t]$. Since $v_0$ is monic in $R_a[t]$, by \cite[Corollary 11.4]{vasershtein1974serre}, it follows that 
 $(v_0, v_1)R_{a}[t]\cap R_{a} + J R_{a}= R_{a}.$ 
Since $A_a = R_a[t]$, so
 we may choose $x \in (v_1,v_2)A \cap R$ such that $xR_a + JR_a = R_a$, i.e.
 \[(x,c_{02},\ldots,c_{m d})\in \Um_{(m + 1)(d-1)+1}(R_a).\]
  Note that $(m + 1)(d-1)+1 \geq 2(d-1) + 1 \geq d + 1 \geq \dim R_a + 2$ since $d\geq 2$ and $\dim R_a \leq d - 1$. Therefore, by Theorem \ref{AG}, we have ${b_{ij}}\in R$ such that
 \[(c_{02}+xa^{m + 1} b_{02},\ldots,c_{md}+xa^{m + 1} b_{md})\in \Um_{(m + 1)(d-1)}(R_a).\]

Since $at \in A$, we have $a^m(b_{0j}+b_{1 j}t+\ldots+b_{m j}t^{m}) \in A$ for $j = 2, \ldots d$. We also have $x \in (v_1,v_2)A$, therefore we have the following : 
\[\v \sim_{\E_{d+1}(A, aA)} \big(v_0, v_1, v_2+ax \{a^m (b_{02}+b_{12}t+\ldots+b_{m 2}t^{m})\}, \ldots, v_d+ax \{a^m(b_{0d}+b_{1d}t+\ldots+b_{m d}t^{m})\}\big)\]

Therefore, by applying suitable $\E_{d+1}(A, aA)$-operations, we may modify $\v$ further such that  the coefficients of $v_2,v_3,\ldots,v_d$ generate $R_{a}$.
By Lemma \ref{RT0}, we apply further $\E_{d+1}(A, aA)$-operations to modify $\v$ such  that $\deg v_1 \leq m$, $v_1$ is $a$-monic and $\max \{\deg v_2, \ldots, \deg v_d\} < m$. Therefore, the proof of the statement of Roitman's machine follows. 
 
By applying Roitman's machine  iteratively, we obtain $\v \sim_{\E_{d+1}(A,aA)} (v_0, w_1, \ldots, w_d)$ such that $w_1$ is a linear $a$-monic polynomial  and $w_i$ is a constant in $aR$ for $i=2,\ldots, d$. 

\noindent
{\bf Step 2 :} 

We know that $\height ((w_dR) + (\sqrt{0} :_R w_d)) \geq 1$ , see \cite[Lemma 2.1]{gupta2016existence}. 
Since $\height (aR) = 1$, an easy argument yields that the set $w_d + a (\sqrt{0}:_R w_d)$ is not contained in any minimal prime ideal of $R$. Therefore, by the prime avoidance lemma, there is $y \in (\sqrt{0} : w_d)$ such that $\height(w_d + ay) \geq 1$.  As $w_d + ay \in aR$, we obtain $\height(w_d + ay) = 1$.
 
Since  $(v_0, w_1, \ldots, w_d) \in \Um_{d + 1}(A, aA)$, there are $p_0, \ldots, p_d \in A$ such that $v_0p_0 + \ldots + w_d p_d = 1$. 
This implies that $ay(v_0p_0 + w_1 p_1+ \ldots + w_{d - 1} p_{d - 1}) = ay - ayw_d p_d$. Therefore, we have 
\begin{align*}
(v_0, w_1, \ldots, w_d) & \sim_{\E_{d+1}(A,aA)} (v_0, w_1, \ldots, w_d + ay - ayw_d p_d) \\
& \sim_{\E_{d+1}(A,aA)} (v_0, w_1, \ldots, w_d + ay) \ \text{by Lemma \ref{26} since $ayw_d p_d \in (a) \cap\Nil A$.}
\end{align*}
This concludes the proof.
\end{proof}

The concept of subintegral extensions was introduced by R. Swan in \cite[\S 2]{swan1980seminormality}.  
The main takeaway from the next two lemmas \ref{gradsub}, \ref{subrel} is that when studying orbit spaces of a graded subring of a polynomial ring $R[t]$ over a domain $R$, one can always assume that the graded subring is birationally equivalent to the polynomial ring itself.   This is achieved by a change of variables and by replacing the subring with an appropriate subintegral extension, a crucial step in the proof of  Lemma \ref{finite}.

\begin{lemma}\label{gradsub}
    Let $R$ be a domain and $A= \oplus_{i\geq 0} A_i$ be a finitely generated graded $R$-subalgebra of $R[t]$. Assume that $A_i\neq 0$ for all $i\geq p \geq 2$. Then there exists a graded $R$-subalgebra $B= \oplus_{i \geq 0} B_i$ satisfying :  
    
    \begin{enumerate}
        \item $A \subset B \subset R[t]$ such that $A_i=B_i$ for all $i\geq p$.
        \item$ B_{p-1}\neq 0$
        \item $B$ is a simple subintegral extension of $A$, i.e. $B=A[u]; u^2,u^3 \in A.$
    \end{enumerate}
    
\end{lemma}
\begin{proof}
    Let $A= R[a_1t^{n_1},a_2t^{n_2},\ldots,a_kt^{n_k}]$ be  a graded $R$-subalgebra of $R[t]$. We choose a sequence of nonzero elements $c_i t^i \in A_i$ for all $i\geq p$. Let $c=c_{n_{1}+p-1}c_{n_{2}+p-1}\ldots c_{n_{k}+p-1}c_{2(p-1)}$ then $ct^{p-1}x \in \oplus_{i\geq p} A_i$ for $x= a_1t^{n_1},\ldots, a_kt^{n_k},ct^{p-1}$. We observe that $(ct^{p-1})^2 \in \oplus_{i\geq p} A_i$ and $ct^{p-1}(\oplus_{i\geq 1} A_i) \subset \oplus_{i\geq p} A_i$. Suppose $B= A[ct^{p-1}] = A + Rct^{p-1}$, then $(ct^{p-1})^2,(ct^{p-1})^3 \in A$, $A_i=B_i$ for all $i\geq p$ and $B_{p-1} \neq 0$. Hence, the proof is complete.
\end{proof}

\begin{lemma}\label{subrel}
Let $R$ be a ring which is not necessarily Noetherian,  $R \hookrightarrow S=R[x]$ be a subintegral extension such that $x^2,x^3 \in R$ and $I$ be an ideal of $R$. Suppose $\maxSpec R$ is a union of finitely many Noetherian subspaces of dimension at most $d$. Then the orbit spaces $\frac{\Um_{d+1}(R,I)}{\E_{d+1}(R,I)}$ and $\frac{\Um_{d+1}(S,IS)}{\E_{d+1}(S,IS)}$ have the group structure. Moreover, the following induced map is an isomorphism of groups: 
\[ \frac{\Um_{d+1}(R,I)}{\E_{d+1}(R,I)}\rightarrow \frac{\Um_{d+1}(S,IS)}{\E_{d+1}(S,IS).}\]
\end{lemma}

\begin{proof}
We begin with an observation that $J=(x^2,x^3)$ is a common ideal of $R$ and $S$. It follows from a result of Swan \cite[Lemma 2.4]{swan1980seminormality} that $R_{x^2}= S_{x^2}$ which implies $D_R(x^2) = D_S(x^2)$ and $$V_R(x^2)= \Spec R/(x^2,x^3)= \Spec S/(x^2,x^3)= V_S(x^2).$$ 
Therefore,  $\maxSpec S$ is a union of finitely many Noetherian subspaces of dimension at most $d$. By \S \ref{pregst}, the orbit spaces $ \frac{\Um_{d+1}(R,I)}{\E_{d+1}(R,I)}$,  $\frac{\Um_{d+1}(S, IS)}{\E_{d+1}(S, IS)}$ have a group structure and the map 
\[\frac{\Um_{d+1}(R, I)}{\E_{d+1}(R, I)}\rightarrow \frac{\Um_{d+1}(S, IS)}{\E_{d+1}(S, IS)}\]
 is a group homomorphism. 
The proof naturally falls into the following two cases.

\noindent
{\bf Case-1 : }($I = R$)
The kernel of the map is trivial by Gubeladze's result, see \cite{gubeladze2001subintegral}, so the map is injective. 
 Let overline denote the image modulo $J$. To prove the surjectivity, we choose $\v \in \Um_{d+1}(S)$. We observe that $S = R + Rx$ and $\overline{x}$ is a nilpotent element of $\overline S$. Therefore, we have  $\overline \v \sim_{\E_{d+1}(S)} \overline \w_1$ for some $\w_1 \in R^n$.  Lifting the elementary operations over $S$, we get $\v \sim_{\E_{d+1}(S)} \w$ where $\w \in \Um_{d+1}(S)$ such that $\w \equiv \w_1 \vpmod J$. A closer look reveals that $\w \in \Um_{d+1}(R)$. Hence, the map is surjective, and the proof follows. 

\noindent
{\bf Case-2 : }($I$ is a proper ideal of $R$)
We recall excision algebras described in \S\ref{preexa}
We observe that \[S \oplus IS = (R + x R) \oplus (I + x I) = (R \oplus I) + (x, 0)(R \oplus I) = (R \oplus I)[(x, 0)].\]
and $(x, 0)^2, (x, 0)^3 \in R \oplus I$. Therefore, $S \oplus IS$ is a subintegral extension of $R \oplus I$. 

We consider the following commutative diagram:
\[
\xymatrix{
  0 \ar[r] & \frac{\Um_{d+1}(R,I)}{\E_{d+1}(R,I)} \ar[d] \ar[r] & \frac{\Um_{d+1}(R\oplus I)}{\E_{d+1}(R\oplus I)} \ar[d] \ar[r] &\frac{\Um_{d+1}(R)}{\E_{d+1}(R)} \ar[d] \ar[r] & 0 \\
  0 \ar[r] &\frac{\Um_{d+1}(S,IS)}{\E_{d+1}(S,IS)} \ar[r] & \frac{\Um_{d+1}(S \oplus IS)}{\E_{d+1}(S \oplus IS)} \ar[r] & \frac{\Um_{d+1}(S)}{\E_{d+1}(S)} \ar[r] & 0
}
\]

The horizontal rows are exact by \S \ref{preexact}. The middle and the right vertical maps are bijective by Case-1. Therefore, the left vertical map is also bijective, and the proof follows. 
\end{proof}

Let $I$ be an ideal of a ring $R$.  If $\v = (v_1, \ldots, v_n) \in \Um_n(R, I)$, we set $\v^{(k)} = (v_1^k, \ldots, v_n), k \in \NN$. 
If $\v \sim_{\E_n(R, I)} \w$, then it follows from a result of Vaserstein \cite[Lemma 4]{vaserstein1986operations} and the Excision theorem \S \ref{preexr} that $\v^{(k)} \sim_{\E_n(R, I)} \w^{(k)}$ for $k \in \NN$.
Therefore, the map $\displaystyle \psi_k : \frac{\Um_{n}(R,I)}{\E_{n}(R,I)} \rightarrow \frac{\Um_{n}(R,I)}{\E_{n}(R,I)}$ given by $ [\v] \xrightarrow{\psi_k}[\v^{(k)}]$ for $k \in \NN$ is well-defined. We introduce the following properties, whose relevance to the topic on hand will be apparent in the last section.

\begin{definition}{\rm (The good property, $k$-divisibility)}\label{defpk}
Let $I$ be an ideal of a ring $R$ such that the orbit space $\frac{\Um_{n}(R,I)}{\E_{n}(R,I)}$ has a group structure for  $n \geq 3$, see \S \ref{pregst}. We say that $\frac{\Um_{n}(R,I)}{\E_{n}(R,I)}$ is good if $[\v^{(k)}] = [\v]^k$ for any $k \in \NN$ and $\v \in \Um_{n}(R,I)$.  The orbit space $\frac{\Um_{n}(R,I)}{\E_{n}(R,I)}$ is called $k$-divisible if the map $[\v] \mapsto [\v]^k$ for $\v \in \Um_{n}(R,I)$ is a surjective map. 
\end{definition}

\begin{remark}\label{suscomp}
Suppose  the orbit space $\frac{\Um_{n}(R, I)}{\E_{n}(R, I)}$ has a group structure for $n \geq 3$, see \S \ref{pregst}.  
If the orbit space $\frac{\Um_{n}(R, I)}{\E_{n}(R, I)}$ is good and is ${(n - 1)!}$-divisible, then
by Suslin's factorial theorem \cite[Theorem 2]{suslin1977stably} and thereafter the result of Rao \cite[Corollary 4.3]{rao2009stably}, one  obtains $\Um_{n}(R, I)= e_1\SL_{n}(R, I)$. 
\end{remark}

The following lemma describes how the good property and the $k$-divisibility for $k \in \NN$ transfer between orbit spaces over rings connected by different homomorphisms.
\begin{lemma}\label{pkhom}
Let $f : R \rightarrow S$ be a ring homomorphism and $I$, $J$ be ideals of $R$, $S$ respectively such that $f(I) \subset J$. Assume that both the orbit spaces $\frac{\Um_n(R, I)}{\E_n(R, I)}$ and $ \frac{\Um_n(S, J)}{\E_n(S, J)}$ have a group structure, see \S \ref{pregst}. Let $\displaystyle f_* : \frac{\Um_n(R, I)}{\E_n(R, I)} \rightarrow \frac{\Um_n(S, J)}{\E_n(S, J)}$ be the map induced by $f$. Let $k \in \NN$. The following holds true.
 \begin{enumerate}
 \item
Suppose $f_*$ is injective and $\v \in \Um_n(R, I)$. If $\frac{\Um_n(S, J)}{\E_n(S, J)}$ is good, then $[\v^{(k)}] = [\v]^k$ for $k \in \NN$. In particular, if $f_*$ is an isomorphism and $\frac{\Um_n(S, J)}{\E_n(S, J)}$ is good ($k$-divisible), then so is $ \frac{\Um_n(R, I)}{\E_n(R, I)}$.

\item
Suppose $[\w] \in \Imj f_*$ for some $\w \in \Um_n(S, J)$. If $\frac{\Um_n(R, I)}{\E_n(R, I)}$ is good, then $[\w^{(k)}] = [\w]^k$ for $k \in \NN$. If $\frac{\Um_n(R, I)}{\E_n(R, I)}$ is $k$-divisible, then $[\w] = [\w']^k$ for some $\w' \in \Um_n(S, J)$. In particular, if $f_*$ is surjective and $ \frac{\Um_n(R, I)}{\E_n(R, I)}$  is good ($k$-divisible), then so is $\frac{\Um_n(S, J)}{\E_n(S, J)}$. 
\end{enumerate}
\end{lemma}

\begin{proof}
Both assertions follow immediately from the equalities $f_*([\v^{(k)}]) =[\w^{(k)}]$ and $f_*([\v]^{k}) =[\w]^{k}$ for $k \in \NN$ and any $\v \in \Um_n(R, I)$, $\w \in \Um_n(S, J)$ satisfying $f([\v])=[\w]$. We skip the details. 
\end{proof}

\section {Orbit spaces over graded $R$-subalgebras of $R[t]$ for $\dim(R) = 2$}
Let $R$ be a ring of dimension $2$ and $A$ be a graded $R$-subalgebra of $R[t]$. This section focuses on studying the orbit space $\frac{\Um_3(A, IA)}{\E_3(A, IA)}$ where $I \subset \rad(R)$ is an ideal of $R$. We recall the Swan-Weibel map, which is essential to the developments in this section.


\subsection{The Swan-Weibel map}\label{SW}
Let $A= \oplus_{i\geq 0} A_i$ be a graded ring such that $A_0=R$. The  Swan-Weibel homomorphism $\psi: A \rightarrow A[t]$ is defined by $\psi(a_0+a_1\cdots+a_r) = a_0+a_1t\cdots+a_rt^r$ for $a_i \in A_i$. Assume that $\phi_m :A[t]\rightarrow A$, $\phi_m(p(t)) =  p(m)$ for $m=0,1$ are the evaluation maps, $\i: R \rightarrow A$ is the inclusion map and $\eta: A\rightarrow A/A_+= R$ is the natural quotient map. Consider the following commutative diagram:

\[
\xymatrix{
    R \ar[rd]_{Id} \ar[r]^{\i} & A \ar[r]^{\psi} \ar[d]_{\eta} 
      \ar@/^1.5pc/[rr]^{Id} & A[t] \ar[d]^{\phi_0} \ar[r]^{\phi_1} & A   \\    
    & R \ar[r]^{\i} & A  &  
}
\]

By Karoubi's result \cite[Corollary 3.2]{vasershtein1974serre}, the natural inclusion $A \rightarrow A[t]$ induces an isomorphism $\W_{\SL}(A) \rightarrow \W_{\SL}(A[t])$. Therefore, the following lemma is a consequence of \cite[Lemma 5.7]{anderson1978projective}.
\begin{lemma}\label{34} 
Let $A = \oplus_{i\geq 0} A_i$ be a graded ring such that $A_0 = R$ and $2R=R$. Then $W_{\SL}(A) \cong W_{\SL}(R)$.
\end{lemma}

Let $\eta_* : \W_{\E}(A) \rightarrow \W_E(R)$ be the map induced by $\eta$. In order to streamline the discussion, we set $\NW_{\E}(A) = \ker \eta_*$.

A careful perusal of  \cite[Proposition 2.4.1]{rao1991completing} by Rao actually gives the following. 
\begin{proposition}\label{Raodiv}
   Let $R$ be a ring such that $2kR= R$ for $k \in \NN$, then $\NW_{\E}(R[t]) = k\NW_{\E}(R[t])$.  
\end{proposition}

We generalize Rao's result to graded subalgebras of $R[t]$ as follows.

\begin{proposition}\label{div}
	Let $R$ be a ring such that $2kR= R$ for $k \in \NN$ and $A$ be a graded $R$-algebra. Then $\NW_{\E}(A) = k \NW_{\E}(A)$. 
	Moreover, if $\maxSpec R$ is a union of finitely many Noetherian subspaces of dimension at most $1$, then $\W_{\E}(A)$ is $k$-divisible.
\end{proposition}

\begin{proof} The following commutative diagram is a consequence of the discussion in  \S \ref{SW}.

\[
\xymatrix{& \NW_{\E}(A)\ar[d]\ar[r]^{\psi_*}& \NW_{\E}(A[t])\ar[d]\\
  \W_{\E}(R) \ar[rd]_{Id} \ar[r]^{i_{*}} & \W_{\E}(A) \ar[d]_{\eta_{*}} \ar[r]^{\psi_{*}} & \W_{\E}(A[t]) \ar[d]^{\phi_{0*}} \ar[r]^{\phi_{1*}} & \W_{\E}(A)   \\
    & \W_{\E}(R) \ar[r]^{i_{*}} & \W_{\E}(A)  & 
}
\]

The diagram yields $\psi_{*}(\NW_{\E}(A)) \subset \NW_{\E}(A[t])$. We have $\NW_{\E}(A[t]) = k \NW_{\E}(A[t])$  by Proposition \ref{Raodiv}.
Therefore, $\psi_{*}(\NW_{\E}(A)) \subset k \NW_{\E}(A[t])$.
Taking images of both sides under $\phi_{1*}$, we get $\NW_{\E}(A)\subset k\W_{\E}(A)$ since $\phi_{1*} \circ \psi_{*}= id_{W_{\E}(A)}$.

  The short exact sequence $0 \rightarrow \NW_{\E}(A) \rightarrow \W_{\E}(A) \xrightarrow{\eta_{*}} \W_{\E}(R) \rightarrow 0$ of abelian groups splits, so there is a surjective homomorphism $\pi_{*} : \W_{\E}(A) \rightarrow \NW_{\E}(A)$ which restricts to the identity on $\NW_{\E}(A)$. Applying $\pi_{*}$ to both sides of $\NW_{\E}(A)\subset k\W_{\E}(A)$, we get $\NW_{\E}(A)\subset k\NW_{\E}(A)$. The proof follows as the reverse inclusion is clear. 

For the last assertion, we observe that under the hypothesis $\W_{\E}(R) = \frac{\Um_3(R)}{\E_3(R)}$ is trivial, see \S \ref{swg}. Therefore, $\NW_{\E}(A)= \W_{\E}(A)$ and consequently $\W_{\E}(A)$ is $k$-divisible. 
\end{proof}

The following lemma and the subsequent remark provide a framework to simplify the proof of the main results.

\begin{lemma}\label{mainprin}
    Let $R$ be a ring of dimension $d$ $(\geq 2)$  and $I \subset \rad (R)$ be an ideal which contains a nonzero divisor $s$.  Let $A =  \oplus_{i\geq 0} I_i t^i$ be a finitely generated graded $R$-subalgebra of $R[t]$. Then the natural map $\displaystyle \frac{\Um_{d + 1}(A, sA)}{\E_{d + 1}(A, sA)} \longrightarrow \frac{\Um_{d + 1}(A, IA)}{\E_{d + 1}(A, IA)}$ is surjective. 
    \end{lemma}

\begin{proof}
Let $\v \in \Um_{d + 1}(A, IA)$.  We need to find $\v' \in \Um_{d + 1}(A, sA)$ such that $\v \sim_{\E_{d + 1}(A, IA)} \v'$.  Let overline denote the image modulo the ideal $sA$. If $I = (x_1, \ldots, x_r)$, we have 
$$\maxSpec \overline{A}  = V_M(\overline{I}) \cup D_M(\overline{I}) = V_M(\overline{I}) \cup \Big(\cup_{i = 1}^r D_M(\overline{x_i})\Big).$$ 
Here $V_M(\overline{I}) = \maxSpec \frac{A}{sA + I}$ and $D_M(\overline{x_i}) \subset \maxSpec (\frac{A}{sA})_{\overline{x_i}}$. 

We know that $\dim A \leq d + 1$, so $\dim A/ sA \leq d$. It is clear that $\overline{x_i} \in \overline{\rad(R)} \subset \rad(\overline{R})$.
Therefore, by Lemma \ref{maindim}, we deduce  $\dim (\frac{A}{sA})_{\overline{x_i}} \leq d -1$ and consequently $\dim [ D_M(\overline{x_i})] \leq d - 1$ for $i = 1, \ldots, r$. 
It follows that $\maxSpec \overline{A} \setminus V_M(\overline{I})$ is a union of $r$ Noetherian subspaces of dimension at most $d - 1$. Applying \cite[Proposition 2.4]{van1983group}, we get $\overline{\v} \sim_{\E_{d + 1}(\overline{A}, \overline{IA})} e_1$. Lifting the elementary operations, we obtain $\v \sim_{\E_{d + 1}(A, IA)} \v'$ for some $\v' \in \Um_{d + 1}(A, sA)$. Hence, the lemma follows. 
\end{proof}

\begin{remark}\label{mainreduct}
Suppose $R$ is a ring of dimension $d \geq 1$ and $I \subset \rad(R)$ is an ideal. We assume that $A$ is a finitely generated graded $R$-subalgebra of the polynomial ring $R[t]$. The orbit space $\frac{\Um_{d + 1}(A, IA)}{\E_{d +1}(A, IA)}$ has an abelian group structure, see Example \ref{prex}.

 Suppose our purpose is to verify if  $\frac{\Um_{d + 1}(A, IA)}{\E_{d + 1}(A, IA)}$ is $k$-divisible for some $k \in \NN$. By Corollary \ref{nzd}, there is an isomorphism $\displaystyle \frac{\Um_{d+1}(A, IA)}{\E_{d+1}(A, IA)} \longrightarrow \frac{\Um_{d+1}(\overline{A},\overline{IA})}{\E_{d+1}(\overline{A},\overline{IA})}$ where overline denotes the image in $A/ \Gamma_{IA} (A)$. As the property of $k$-divisibility transfers along an isomorphism of orbit spaces, it is enough to show that $\frac{\Um_{d+1}(\overline{A},\overline{IA})}{\E_{d+1}(\overline{A},\overline{IA})}$ is $k$-divisible. 
If $\dim (R/ \Gamma_I(R))< d$, then $\frac{\Um_{d+1}(\overline{A},\overline{I A})}{\E_{d+1}(\overline{A},\overline{I A})}$  is trivial, see \cite[Theorem 3.5]{garggupta2025graded}. If $\overline{I} =0$, then also $\frac{\Um_{d+1}(\overline{A},\overline{I A})}{\E_{d+1}(\overline{A},\overline{I A})}$ is trivial. A trivial orbit space is $k$-divisible, so the only case of interest is when $\dim (R / \Gamma_I(R)) = d$ and $\overline{I} \neq 0$. We have inclusions of rings $R/ \Gamma_I(R) \hookrightarrow A/ \Gamma_{IA} (A) \hookrightarrow R/ \Gamma_I(R) [t]$ and $\overline{I} \subset \rad(\overline{R})$. Therefore, it suffices to assume that $I$ contains a nonzero divisor. 

Due to Lemmas \ref{pkhom} and \ref{mainprin}, it is enough to further assume that $I$ is a proper principal ideal generated by a nonzero divisor. We have inclusions $R_{red} \hookrightarrow A_{red} \hookrightarrow R_{red}[t]$ of rings, where $R_{red} = R/ \Nil R$, $A = A/ \Nil A$; and a bijection $\displaystyle \frac{\Um_{d + 1}(A, IA)}{\E_{d + 1}(A, IA)} \longrightarrow \frac{\Um_{d + 1}(A_{red}, IA_{red})}{\E_{d + 1}(A_{red}, IA_{red})}$ by Lemma \ref{26}. The property of being nonzero divisors is retained in the passage to the quotient ring $R_{red}$. 

Hence, in our quest of the property of $k$-divisibility of $\frac{\Um_{d + 1}(A, IA)}{\E_{d + 1}(A, IA)}$, we may say that it is enough to assume further that $R$ is a reduced ring and $I \subset \rad (R)$ is a proper principal ideal generated by a nonzero divisor. An argument verbatim as above enables us to say the same when our interest is to verify if the orbit space $\frac{\Um_{d + 1}(A, IA)}{\E_{d + 1}(A, IA)}$ is good. 
\end{remark}
The subsequent lemma serves as a further aid in simplifying the proofs of ensuing results.
\begin{lemma} \label{47}
Let $R$ be a ring of dimension $d (\geq 2)$, $A$ be a graded $R$-subalgebra of $R[t]$ and $I =\oplus_{i\geq 0} I_i$ be a homogeneous ideal of $A$ such that  $I_0 \subset \rad (R)$. Then for any $\v \in \Um_{d + 1}(A, I)$ there exists a finitely generated graded $R$-subalgebra $B = \oplus_{i \geq 0}B_i$ of $A$ such that  $B_+ = \oplus_{i \geq 1} B_i \subset I$ and $\v \sim_{\E_{d + 1}(A, I)} \v'$ for some $\v' \in \Um_{d + 1}(B, B_{+})$.
In particular $[\v]$ is contained in the image of $\displaystyle \frac{\Um_{d + 1}(B, B_{+})}{\E_{d + 1}(B, B_{+})} \longrightarrow \frac{\Um_{d + 1}(A, I)}{\E_{d + 1}(A, I)}$. 
\end{lemma}
\begin{proof}
 Let $\v \in  \Um_{d+1}(A, I)$. Suppose overline denotes the image modulo the ideal $I_+ = \oplus_{i\geq 1} I_i$. Then $\overline{\v} \in \Um_{d+1}(R, I_0)$. We have $\overline{\v} \sim_{\E_{d+1}(R, I_0)} e_1$ since $I_0 \subset \rad (R)$. Lifting the elementary operations, we get $\v \sim_{\E_{d+1}(A, I)} \v'$ for some $\v' = (v_0', \dots , v_d') \in \Um_{d+1}(A, I_{+})$.
We choose $(w_0', \ldots, w_d')  \in \Um_{d+1}(A, I_{+})$ such that $\sum_{i\geq 0}^{d}v_i'w_i'=1$. 

Let $B= \oplus_{i \ge 0} B_i$ be the graded $R$-subalgebra of $A$ generated by all homogeneous components of the coordinates $v_i',  w_i'$ for $i = 1, \ldots, d$. It is easily seen that $B_{+} = \oplus_{i >0} B_i \subset I_{+}$, $\v' \in \Um_{d + 1}(B, B_{+})$ and the natural map $\displaystyle \frac{\Um_{d + 1}(B, B_{+})}{\E_{d + 1}(B, B_{+})} \longrightarrow \frac{\Um_{d + 1}(A, I)}{\E_{d + 1}(A, I)}$ sends $[\v']$ to $[\v]$. Hence, the proof follows. 
\end{proof}

 We are now prepared to prove the principal result of this section.
\begin{proposition}\label{42} 
	Let $R$ be a ring of dimension $2$ such that $2kR= R$ for $k \in \NN$, $I \subset \rad(R)$ be an ideal of $R$ and $A$ be a finitely generated graded $R$-subalgebra of $R[t]$. Then the orbit space $\frac{\Um_{3}(A,IA)}{\E_{3}(A, IA)}$ is good and $k$-divisible. 
\end{proposition}

\begin{proof}
As explained in Example \ref{prex}, the orbit space $\frac{\Um_{3}(A, IA)}{\E_3(A, IA)}$ has a group structure. On account of Remark \ref{mainreduct}, it is enough to assume that $R$ is a reduced ring and $I \subset \rad(R)$ is a principal ideal generated by a nonzero divisor $s$. Let $\v \in \Um_3(A, IA)$. According to Lemma \ref{47}, there is a finitely generated graded $R$-subalgebra $B$ of $A$ such that $[\v]$ is contained in the image of $\displaystyle \frac{\Um_{3}(B, B_{+})}{\E_{3}(B, B_{+})} \rightarrow \frac{\Um_{3}(A, I)}{\E_{3}(A, I)}$. Since $\v$ is arbitrary, it suffices to prove that $\frac{\Um_{3}(B,B_+)}{\E_3(B,B_+)}$ is good and $k$-divisible due to Lemma \ref{pkhom}.

In our setup, $\maxSpec(R) = V_M(s) \cup D_M(s)$ is a union of two subspaces of dimension at most $1$. In contrast, $\maxSpec(B) = V_M(sB) \cup D_M(s)$ is a union of two subspaces of dimension at most $2$. It follows that $\frac{\Um_{3}(R)}{\E_{3}(R)}$ is trivial and  $\displaystyle \frac{\Um_{3}(B, B_{+})}{\E_{3}(B, B_{+})} \rightarrow \frac{\Um_{3}(B)}{\E_{3}(B)}$  is an isomorphism by Lemma \ref{retract}. Therefore, the proof is completed once we show that $\frac{\Um_{3}(B)}{\E_3(B)}$ is good and $k$-divisible.

 Let $\u = (u_0, u_1, u_2) \in \Um_{3}(B)$. We have $\frac{\Um_{3}(B)}{\SL_3(B)} \cong \W_{\SL}(B)$ and $\frac{\Um_{3}(B)}{\E_{3}(B)}\cong \W_{\E}(B)$, see \S \ref{swg}.
 Applying Lemma \ref{34}, we conclude that $\displaystyle \frac{\Um_{3}(B)}{\SL_{3}(B)} \cong\W_{\SL}(B) \cong \W_{\SL}(R) \cong \frac{\Um_{3}(R)}{\SL_{3}(R)}$ is trivial, consequently $\Um_{3}(B)=e_1\SL_3(B)$. 
 It follows from \cite[Lemma 1.5.1]{rao1988bass} that $\u \sim_{\E_3(B)} (-u_0, u_1, u_2)$. Therefore, by \cite[Lemma 1.3.1]{rao1988bass}, we have  $[\u]^r =[\u^{(r)}]$ for any $r \in \NN$, i.e. $\frac{\Um_{3}(B)}{\E_{3}(B)}$ is good.  It is $k$-divisible since so is  $\W_{\E}(B) \cong \frac{\Um_{3}(B)}{\E_3(B)}$ by Proposition \ref{div}.  Hence, the proof is completed. 
\end{proof}

\section{Main results}
 This section is dedicated to proving our main theorem. Our approach begins by establishing a series of lemmas that lay the groundwork for proving a generalization of Proposition \ref{42} in the previous section. The first lemma draws its inspiration from the Artin-Rees lemma.

\begin{lemma}\label{43} 
Let $A =\oplus_{i\geq } I_it^i$ be a finitely generated graded $R$-subalgebra of $R[t]$. Let $I$, $J$ be ideals of $R$ and $I \subset I_1$. Then there exists $k \in \N$ such that $I^nA \cap JR[t] \subset JI^{n-k}A$ for all $n\geq k$.
\end{lemma}
\begin{proof}
We consider the subalgebra $B= A[Iu]$  of the polynomial algebra $R[t,u]$. Then $B= \oplus_{i,j\geq 0} B_{ij}$ is a $\Z_{\geq 0}\times \Z_{\geq 0}$-graded ring where $B_{ij}=I_iI^jt^iu^j$. Since $B$ is a finitely generated algebra over $A$ and $A$ is Noetherian, it follows that $B$ is also Noetherian. Consider the ideal $L$ of $B$ given by $L= B \cap JR[t,u] = \oplus_{i,j\geq 0} L_{ij}$, where $L_{ij}= (I_iI^j\cap J)t^iu^j$. Since $B$ is Noetherian, $L$ is a finitely generated ideal of $B$. Suppose $L$ is generated by $g_1,g_2,\ldots ,g_r$ where each $g_i$ is a homogeneous element of degree $(k_i,l_i); i=1,2,\ldots,r$. Let $k= \max_{1\leq i\leq r}\{l_i+k_i\}$. 

Let $\alpha \in I_m I^n  \cap J$ for $n\geq k$, $m\in \Z_{\geq 0}$, then $\alpha t^mu^n \in L_{mn}$. We have $\alpha t^m u^n = \sum_{i=1}^{r}q_ig_i$ where $q_i=c_it^{m-k_i}u^{n-l_i}$ and $c_i\in I_{m-k_i}I^{n-l_i}$. As the coefficients of each $g_i$ are in $J$, it follows that 
$$\alpha \in \sum_{i=1}^{r} JI_{m-k_i}I^{n-l_i} = \sum_{i=1}^{r}( JI_{m-k_i}I^{k_i})I^{n-(l_i+k_i)} \subset J I_m I^{n-k},$$ 
since $I_{m-k_i}I^{k_i} \subset I_m$ and $I^{n-(l_i+k_i)} \subset I^{n - k}$. 
Therefore, $I^n I_m \cap J \subset J I^{n-k}I_m$ for $n\geq k$, $m\in \Z_{\geq 0}$, consequently $I^nA \cap JR[t] \subset JI^{n-k}A$ for all $n\geq k$ and the proof follows. 
\end{proof}

\begin{lemma}\label{pthick}
Let $I$ be an ideal of a ring $R$ and $\v  \in R^{n}$, $n\geq 3$ be such that $\v \equiv e_1 \vpmod I $. Then, for any $N\in \NN$ there exists an $\alpha \in \E_{n}(R,I)$ such that $\v \alpha = \w$  where $\w \equiv e_1 \vpmod {I^N }$.
\end{lemma}

\begin{proof}
Let overline denote the image modulo the ideal $I^N$. Now $\overline{\v} \in \Um_{n}(\overline{R}, \overline{I})$ and by Lemma \ref{nil}, we find $\beta \in \E_{n}(\overline{R}, \overline{I})$ such that $\overline{\v} \beta = e_1$. Any preimage $\alpha \in \E_{n}(R,I)$ of $\beta$ has the desired property.
\end{proof}
The forthcoming lemma and the subsequent remark are instrumental in facilitating the induction step of Proposition \ref{RSL}, thereby advancing the proof.

\begin{lemma}\label{imp}
Let $R$ be a ring and $A = \oplus_{i\geq 0} I_it^i$ be a finitely generated graded $R$-subalgebra of $R[t]$ such that $I_1$ contains a nonzero ideal $I$ of $R$. Let $d \geq 2$ and $\v =(v_0, \ldots, v_d)\in \Um_{d + 1}(A, IA)$ be such that $v_d \in I \subset R$. Let overline denote the image modulo the ideal $v_dR[t]\cap A$ of $A$. 
Then for sufficiently large $l \in \NN$, there is a map 
\[\theta: \frac{\Um_{d}(\overline{A}, I^{l}\overline A)}{\E_{d}(\overline{A},  I^l\overline A)}\rightarrow \frac{\Um_{d+1}(A,IA)}{\E_{d+1}(A,IA)} \ \text{defined by} \  [(\overline{a_0}, \ldots, \overline{a_{d-1}})] \mapsto [(a_0, \ldots, a_{d-1}, v_d)],\] 
where $(\overline{a_0}, \ldots, \overline{a_{d-1}}) \in \Um_{d}(\overline{A}, I^{l}\overline A)$ and
$(a_0, \ldots, a_{d - 1}) \in A^d$ is a lift of $(\overline{a_0}, \ldots, \overline{a_{d-1}})$ such that \\
$(a_0, \ldots, a_{d-1}) \equiv e_1\vpmod {I^{l}A}$. Moreover, $\Imj \theta$ contains $[\v]$.
\end{lemma}

\begin{proof}
According to Lemma \ref{43}, there exists \( l \in \mathbb{N} \) such that
$I^l A \cap v_d R[t] \subset I v_d A$. We first construct a map $$\theta_0 : \displaystyle \Um_{d}(\overline{A}, I^{l}\overline{A}) \rightarrow  \frac{\Um_{d + 1}(A, IA)}{\E_{d + 1}(A, IA)}.$$ 

Let $(\overline{a_0}, \ldots, \overline{a_{d - 1}}) \in \Um_{d}(\overline{A}, I^{l}\overline{A})$ where $(a_0, \ldots, a_{d - 1}) \in A^d$ satisfies ${(a_0, \ldots, a_{d-1}) \equiv e_1 \pmod {I^{l}A}}$. We choose $(\overline{a_0'}, \ldots, \overline{a_{d - 1}'}) \in \Um_{d}(\overline{A}, I^{l}\overline{A})$ for $(a'_0, \ldots, a'_{d - 1}) \in A^d$ satisfying $(a'_0, \ldots, a'_{d-1}) \equiv e_1 \vpmod {I^{l}A}$ such that $\sum_{i = 0}^{d-1} \overline{a_i} \overline{a_i'} = 1$.  It follows that  $\sum_{i = 0}^{d-1} {a_i} {a_i'} - 1 \in I^{l} A \cap v_d R[t] \subset I v_d A$. Consequently, we deduce $(a_0, \ldots, a_{d - 1}, v_{d}) \in \Um_{d + 1}({A}, I{A})$. We define
$(\overline{a_0}, \ldots, \overline{a_{d-1}}) \xmapsto{\theta_0} [(a_0, \ldots, a_{d-1}, v_d)]$.

To show that $\theta_0$ is well defined, let
$$(\overline{a_0}, \ldots, \overline{a_{d-1}}) = (\overline{b_0}, \ldots, \overline{b_{d-1}}) \in \Um_{d}(\overline{A}, I^{l}\overline{A}) \ \text{for}\ (a_0, \ldots, a_{d - 1}),  (b_0, \ldots, b_{d - 1})\in A^d$$ 
satisfying $(a_0, \ldots, a_{d-1}) \equiv (b_0, \ldots, b_{d-1}) \equiv e_1 \vpmod {I^{l}A}$.
Then 
$a_i - b_i \in I^{l} A \cap v_d R[t] \subset Iv_d A,$ so $(a_0, \ldots, a_{d-1}, v_d) \sim_ {\E_{d + 1}(A, IA)}(b_0, \ldots, b_{d-1}, v_d)$. Therefore, $\theta_0$ is well-defined. 

Let $\v_1, \v_2 \in \Um_{d}(\overline{A}, I^{l}\overline{A})$ such that $\v_1 \sim_{\E_{d}(\overline{A}, I^{l}\overline{A})} \v_2$. Then there is an $\varepsilon \in \E_{d}(A, I^{l}A)$ such that $\v_1 \overline{\varepsilon} = \v_2$. Let $\widetilde{\v_1} \in A^d$ be a lift of $\v_1$ such that $\widetilde{\v_1} \equiv e_1 \vpmod {I^{l}A}$. Then $\widetilde{\v_1}\varepsilon$ is a lift of $\v_2$ satisfying $\widetilde{\v_1} \varepsilon \equiv e_1 \vpmod {I^{l}A}$. It follows that
\[\theta(\v_1) = [(\widetilde{\v_1}, v_d)] = \Big[(\widetilde{\v_1}, v_d) \begin{pmatrix} \varepsilon & 0\\ 0 & 1 \end{pmatrix}\Big] = [(\widetilde{\v_1}\varepsilon, v_d)] = \theta(\v_2). \]
Therefore, $\theta_0$ induces the map $\theta$ as stated in the hypothesis.

By Lemma \ref{pthick}, we find an elementary matrix of the form $\alpha = \alpha_0 \perp 1$ for $\alpha_0 \in \E_d(A, I A)$ such that $\v \alpha = \v'$, where $\v' = (v_0', \ldots, v_{d-1}', v_d) \in \Um_{d+1}(A, IA)$ satisfies $(v_0', \ldots, v_{d-1}') \equiv e_1 \vpmod {I^{l}A}$. It is clearly seen that $\theta$ maps $[(\overline{v_0'}, \ldots, \overline{v_{d-1}'})]$ to $[\v'] = [\v]$. 
Hence, the proof is complete. 
 \end{proof}

\begin{remark}\label{magh}
In the setup of Lemma \ref{imp}, assume further that $I \subset \rad (R)$, $\height (v_d) = 1$ and $\dim (R) = d$. Then both $\frac{\Um_{d}(\overline{A}, I^{l}\overline A)}{\E_{d}(\overline{A},  I^l \overline {A})}$ and $\frac{\Um_{d+1}(A,IA)}{\E_{d+1}(A,IA)}$ have the structure of an abelian group according to Example \ref{prex}. Let $[\v_1], [\v_2] \in \frac{\Um_{d}(\overline{A}, I^{l}\overline A)}{\E_{d}(\overline{A},  I^l \overline {A})}$ for $\v_1, \v_2\in \Um_{d}(\overline{A}, I^{l}\overline{A})$. We choose $(a, a_1 \ldots, a_{d - 1}), (b, a_1 \ldots, a_{d - 1}) \in A^{d + 1}$ satisfying $(a, a_1, \ldots, a_{d-1}) \equiv (b, a_1, \ldots, b_{d-1}) \equiv e_1 \vpmod {I^{l}A}$ such that $\v_1 = (\overline{a}, \overline{a_1}, \ldots, \overline{a_{d-1}})$ and $\v_2 = (\overline{b}, \overline{a_1}, \ldots, \overline{a_{d-1}})$. 
We saw before that $(a, a_1 \ldots, a_{d - 1}, v_d) \in \Um_{d + 1}({A}, I{A})$.
 We pick $p \in A$ such that $p \equiv 1 \vpmod {I^l A}$  and  $ap \equiv 1 \vpmod {(a_1, \ldots, a_{d - 1}, v_d)}$. Then $\overline{a} \ \overline{p} \equiv 1 \vpmod {(\overline a_1, \ldots, \overline {a_{d-1}})}$. We deduce $[\v_2]*[\v_1] = [((\overline b + \overline p)\overline{a} - 1, (\overline b + \overline p)\overline{a_1}, \ldots, \overline{a_{d-1}})]$ and 
\begin{align*}
\theta([\v_2]*[\v_1]) & = [(( b +  p)a - 1 , (b + p)a_1, \ldots, a_{d-1}, v_d)]\\
                              & = [(b, a_1 \ldots, a_{d - 1}, v_d)]*[(a, a_1 \ldots, a_{d - 1}, v _d)]
                              = \theta([\v_2]) * \theta([\v_1]), \ \text{see (\ref{gop}) in \S \ref{pregst}}.
\end{align*}
Therefore, $\theta$ is a group homomorphism. Moreover, $\theta([\v_1^{(n)}]) = [(a, a_1 \ldots, a_{d - 1}, a_d)^{(n)}]$ for $n \in \NN$.
\end{remark}

 The next result deals with a particular case of the main theorem.

\begin{proposition}\label{RSL}
Let $R$ be a ring of dimension $d$ $(\geq 2)$  and $2kR = R$ for some $k \in \NN$. Let $A = \oplus_{i\geq 0} I_it^i$ be a finitely generated graded $R$-subalgebra of $R[t]$ such that $I_1 \cap \rad(R)$ contains a nonzero ideal $I$. Then the orbit space $ \frac{\Um_{d + 1}(A, IA)}{\E_{d + 1}(A, IA)}$  is good and $k$-divisible. 
\end{proposition}

\begin{proof} 
The orbit space $ \frac{\Um_{d + 1}(A, IA)}{\E_{d + 1}(A, IA)}$ has a group structure according to Example \ref{prex}. We prove the result by induction on $d$. When $d = 2$, the statement follows from Proposition \ref{42}, so we assume $d \geq 3$. 
 
On account of Remark \ref{mainreduct}, it is enough to assume that $R$ is a reduced ring and $I$ is a proper principal ideal generated by a nonzero divisor $s$. Here, one observes that in the passage of reduction described in the remark, the condition $I \subset I_1$ is retained. 

Let $\v \in \Um_{d+1}(A, sA)$. Applying  Lemma \ref{RT}, we have $ \v \sim_{\E_{d+1}(A,sA)} \v'$, where  
$$\v' = (w_0, w_1, a_2, \ldots, a_d) \in \Um_{d+1}(A, sA)$$ such that $a_2, \ldots, a_d \in sR$ and $\height(a_dR) = 1$. Let overline denote the image modulo the ideal $a_d R[t] \cap A$. 

For  sufficiently large $l \in \NN$, we have a group  homomorphism $\displaystyle \theta : \frac{\Um_{d}(\overline{A}, s^{l + 1}\overline{A})}{\E_{d}(\overline{A}, s^{l + 1}\overline{A})} \rightarrow  \frac{\Um_{d + 1}(A, sA)}{\E_{d + 1}(A, sA)}$  such that image of $\theta$ contains $[\v'] = [\v]$, see Lemma \ref{imp} and Remark \ref{magh}. Let $\theta([\w]) = [\v]$ for $\w \in \Um_{d}(\overline{A}, s^{l + 1}\overline{A})$. It is easily seen that $\theta([\w^{(n)}]) = [\v^{(n)}]$ for all $n \in \NN$. 

Now $\overline A$ is a graded $R/(v_d)$-subalgebra of $R/(v_d)[t]$ and  $\dim R/ (v_d) = d -1$, so by induction hypothesis $\frac{\Um_{d}(\overline{A}, s^{l + 1}\overline{A})}{\E_{d}(\overline{A}, s^{l + 1}\overline{A})}$ is good and $k$-divisible. Therefore, $[\w^{(n)}] = [\w]^n$ for $n \in \NN$. Applying $\theta$ to both sides, we get $[\v^{(n)}] = [\v]^n$ for $n \in \NN$. Therefore, $\frac{\Um_{d + 1}(A, sA)}{\E_{d + 1}(A, sA)}$ is good. 
 
 We have that  $\Imj \theta$ is $k$-divisible and $[\v] \in \Imj \theta$, so we find a $\v' \in \Um_{d + 1}(A, sA)$ such that $[\v] = k [\v']$. Therefore, $\frac{\Um_{d + 1}(A, sA)}{\E_{d + 1}(A, sA)}$ is $k$-divisible. The proof is completed since the induction holds. 
 \end{proof}
 
Our subsequent proposition is established through induction. For ease of exposition, the base case is treated separately in the following form.
 
\begin{lemma}\label{finite}
    Let $R$ be a domain of dimension $d$ $(\geq 2)$   and $I \subset \rad(R)$ be an ideal of $R$.  Assume  $2kR=R$ for some $k \in \NN$. Suppose $A$ is a finitely generated graded $R$-subalgebra of $R[t]$. Then $\frac{\Um_{d+1}(A, IA)}{\E_{d+1}(A,IA)}$ is good and $k$-divisible.
\end{lemma}

\begin{proof}
We may assume that $I \neq 0$ as otherwise, the proof is trivial. Let $A= R[a_1t^{n_1},a_2t^{n_2},\ldots,a_kt^{n_k}]$ and $\gcd(n_1,\ldots,n_k)= g$. We may assume that $g = 1$ because otherwise, we can view $A$ as a graded $R$-subalgebra of $R[t^g]$ and complete the proof replacing $t^g$ by $t$. Since $g = 1$, there is an integer $p$ such that the integers $p, p + 1, \ldots$ are in the numerical semigroup generated by $n_1, \ldots, n_k$. Therefore, $A_i \neq 0$ for all $i \geq p$. 
By Lemma \ref{gradsub}, we have a suitable subintegral extension $A \hookrightarrow B$ where $B= \oplus_{i \geq 0} B_i$ is a finitely generated graded $R$-subalgebra of $R[t]$ such that $B_1 \neq 0$. The natural map 
$$ \frac{\Um_{d+1}(A, IA)}{\E_{d+1}(A,IA)} \rightarrow \frac{\Um_{d+1}(B, IB)}{\E_{d+1}(B,IB)}$$ 
is an isomorphism due to Example \ref{prex} and Lemma \ref{subrel}. According to Lemma \ref{pkhom}, we may replace $A$ by $B$ to assume that $I_1 \neq 0$.

We choose $s \in I \cap I_1$, $s \neq 0$. It is enough to show that $\frac{\Um_{d+1}(A, sA)}{\E_{d+1}(A, sA)}$ is good and $k$-divisible due to Lemmas \ref{pkhom} and \ref{mainprin}. The proof now follows from Proposition \ref{RSL}.
\end{proof}
 We now proceed to prove the proposition.
 
\begin{proposition}\label{maint}\label{57}
    Let $R$ be a ring of dimension $d$ $(\geq 2)$ , $I \subset \rad(R)$ be any ideal of $R$ and $A$ be a  finitely generarated graded $R$-subalgebra of $R[t]$. Assume $2kR=R$ for some $k \in \NN$. Then $\frac{\Um_{d+1}(A,IA)}{\E_{d+1}(A, IA)}$ is good and $k$-divisible. 
 \end{proposition}

\begin{proof}
We begin by noting that, according to Example \ref{prex}, the orbit space $\frac{\Um_{d+1}(A,IA)}{\E_{d+1}(A, IA)}$ has a group structure. It is enough to assume that $R$ is reduced and $I$
is a proper principal ideal generated by a nonzero divisor $s$ according to Remark \ref{mainreduct}.  We prove the result by induction on the number of minimal prime ideals $\gamma(R)$ of $R$. If $\gamma(R) = 1$, then $R$ is an integral domain,  so the result follows by Lemma \ref{finite}.

Let $\gamma(R) = n \geq 2$. Let $\p_1,\ldots,\p_n$ be the set of minimal prime ideals of $R$. We define $L = \p_1 \cap \ldots \cap \p_{n - 1}$ and $K = \p_n$.
We choose $s_1 \in (I \cap L) \setminus \p_n$ and $s_2 \in (I \cap K) \setminus \p_1 \cup \p_2\ldots\cup \p_{n-1} $. Since $R$ is reduced, we have $KL = 0$, so $s_1s_2= 0$. By  Lemma \ref{411},
    we have a bijection : 
     \[\frac{\Um_{d+1}(A,(s_1,s_2))}{\E_{d+1}(A,(s_1, s_2))} \rightarrow \frac{\Um_{d+1}(\frac{A}{s_2A},\overline{s_1})}{\E_{d+1}(\frac{A}{s_2A},\overline{s_1})} \times \frac{\Um_{d+1}(\frac{A}{s_1A},\overline{s_2})}{\E_{d+1}(\frac{A}{s_1A},\overline{s_2})}. \] 

 Note that $s_1(KR[t]\cap A)=0$ and  $s_2(LR[t]\cap A)=0$. 
     Therefore, Lemma \ref{IJ} yields the following isomorphisms:
    $$\frac{\Um_{d+1}(\frac{A}{s_2A},\overline{s_1})}{\E_{d+1}(\frac{A}{s_2A},\overline{s_1})} \longrightarrow \frac{\Um_{d+1}(\frac{A}{KR[t]\cap A},\overline{s_1})}{\E_{d+1}(\frac{A}{KR[t]\cap A},\overline{s_1})},$$
    $$\frac{\Um_{d+1}(\frac{A}{s_1A},\overline{s_2})}{\E_{d+1}(\frac{A}{s_1A},\overline{s_2})} \longrightarrow \frac{\Um_{d+1}(\frac{A}{LR[t]\cap A},\overline{s_2})}{\E_{d+1}(\frac{A}{LR[t]\cap A},\overline{s_2})}.$$ 

We observe that $\gamma(R/ K) = 1$ and $\gamma(R/ LR) = n - 1$. By the induction hypothesis both orbit spaces on the right side are good and $k$-divisible, so the orbit space $\frac{\Um_{d+1}(A,(s_1,s_2))}{\E_{d+1}(A,(s_1,s_2))}$ is good and $k$-divisible.
    
Now $u = s_1 + s_2$ is a nonzero divisor in $I$.  The natural map 
$\displaystyle \frac{\Um_{d + 1}(A, uA)}{\E_{d + 1}(A, uA)} \rightarrow \frac{\Um_{d + 1}(A, IA)}{\E_{d + 1}(A, IA)}$
 is surjective by Lemma \ref{mainprin}. As this map factors through $ \frac{\Um_{d + 1}(A, (s_1, s_2)A)}{\E_{d + 1}(A, (s_1, s_2)A)}$, the natural map 
 $$\frac{\Um_{d + 1}(A, (s_1, s_2)A)}{\E_{d + 1}(A, (s_1, s_2)A)} \rightarrow \frac{\Um_{d + 1}(A, IA)}{\E_{d + 1}(A, IA)}$$ is also surjective. Therefore, the proof follows by Lemma \ref{pkhom}. 
\end{proof}

In the context of the proposition below, its statement implies that any $\v \in \Um_{d+1}(R[t], (t)) $ can be transformed, via elementary operations, into $\w^{(k)}$, for $\w \in \Um_{d+1}(R[t], (t))$. This was previously observed by Rao in the case where $R$ is a local ring, see \cite[Corollary 2.3]{rao1988question}.

\begin{proposition}\label{mgkd}
Let $R$ be a ring of dimension $d$ $(\geq 2)$  such that $2k R = R$ for some $k \in \NN$. Then $ \frac{\Um_{d + 1}(R[t], (t))}{\E_{d + 1}(R[t], (t))}$ is good and $k$-divisible. 
\end{proposition}

\begin{proof}

We first justify that the orbit space $\frac{\Um_{d + 1}(R[t], (t))}{\E_{d + 1}(R[t], (t))}$ has an abelian group structure. If $d = 2$, then 
\begin{align*}
\SL_4(R[t], (t)) \cap \E( R[t], (t) ) &= \SL_4(R[t], (t)) \cap \E_4( R[t]) \ \text{by \cite[Theorems 7.8]{suslin1977structure}} \\
& = \E_4(R[t], (t)) \  \text{by \cite[p. 204, Lemma 1.5]{lam2006serre}}.
\end{align*}
Therefore, by \cite[Theorem 5.3]{gupta2014nice} it follows that $ \frac{\Um_{d + 1}(R[t], (t))}{\E_{d + 1}(R[t], (t))}$ is an abelian group under the operation as in (\ref{gop}), \S \ref{pregst}. If $d \geq 3$, the ensuing  inequality $d \geq \frac{d + 1}{2} + 1$ implies that $\frac{\Um_{d + 1}(R[t], (t))}{\E_{d + 1}(R[t], (t))}$ is an abelian group under the same operation, see \S \ref{pregst}. 

 We may assume that $R$ is a reduced ring. We choose $\v_0 \in \Um_{d + 1}(R[t], (t))$. Let $S$ be the multiplicatively closed subset of nonzero divisors of $R$. Then $S^{-1}R$ is a product of finitely many fields, so ${\v_0}_S \sim_{\E_{d + 1}(R_S[t], (t))} e_1$. We choose $s \in S$ such that ${\v_0}_s \sim_{\E_{d + 1}(R_s[t], (t))} e_1$.  Since $R = R_s \times_{R_{s(1 + sR)}} R_{(1 + sR)}$, we have the following map:

 \begin{equation}\label{meq1}
 \frac{\Um_{d + 1}(R[t], (t))}{\E_{d + 1}(R[t], (t))} \xrightarrow{\quad \Psi \quad} \frac{\Um_{d + 1}(R_s[t], (t))}{\E_{d + 1}(R_s[t], (t))} \times \frac{\Um_{d + 1}(R_{(1 + sR)}[t], (t))}{\E_{d + 1}(R_{(1 + sR)}[t], (t))}=G, \hspace{0.5cm} [\v]\xmapsto{\Psi} ([\v_s],[\v_{1+sR}]) 
 \end{equation}

We show that $\Psi$ is an isomorphism. It is clearly seen that $\psi$ is a group homomorphism. It follows from the local-global principle \cite[Theorem 2.3]{rao1985elementary} that $\Psi$ is injective.
 Note that $\dim R_{s(1 + sR)} \leq d - 1$, so the orbit space $H= \displaystyle \frac{\Um_{d + 1}(R_{s(1 + sR)}[t], (t))}{\E_{d + 1}(R_{s(1 + sR)}[t], (t))}$ is trivial. To prove that $\Psi$ is surjective, consider $([\u],[\w])\in G$ for $\u \in \Um_{d + 1}(R_s[t], (t))$ and $\w \in \Um_{d + 1}(R_{(1 + sR)}[t], (t))$. We observe that $[\u_{1+sR}]=[\w_{s}]=[e_1]$ as $H$ is trivial, so we can find a $\sigma\in \E_{d + 1}(R_{s(1 + sR)}[t], (t))$ such that  $\u_{1+sR} \sigma= \w_s$. Applying Quillen's splitting lemma \cite[Lemma 3.5]{suslin1977structure}, we get elementary matrices $\sigma_1 \in \E_{d+1}(R_{s}[t],(t))$ and  $\sigma_2 \in \E_{d+1}(R_{1+sR}[t],(t))$ such that $\sigma= (\sigma_1)_{1+sR} \circ  (\sigma_2)_{s} $. It follows that $(\u \sigma_1)_{1+sR}= (\w\sigma_2^{-1})_s$. By standard patching argument, we can find a $\v \in \Um_{d + 1}(R[t], (t))$ such that $\v_s = \u \sigma_1$ and $\v_{1 + sR} = \w \sigma_2^{-1}$. Therefore, $$\Psi([\v])=([\v_s], [\v_{1+sR}])=([\u \sigma_1], [\w\sigma_2^{-1}])=([\u], [\w]).$$
  
We next prove that  $\displaystyle \frac{\Um_{d + 1}(R_{1 + sR}[t], (t))}{\E_{d + 1}(R_{1 + sR}[t], (t))}$ is good and $k$-divisible. The orbit space $\displaystyle  \frac{\Um_{d + 1}(R_{1 + sR})}{\E_{d + 1}(R_{1 + sR})}$ is trivial, since $\maxSpec(R_{1 + sR}) = V_M(s) \cup D_M(s)$ is a union of two subspaces of dimension at most $d - 1$. By Lemma \ref{retract}, we deduce that 
 \begin{equation}\label{meq2}
\frac{\Um_{d + 1}(R_{1 + sR}[t], (t))}{\E_{d + 1}(R_{1 + sR}[t], (t))} \cong \frac{\Um_{d + 1}(R_{1 + sR}[t])}{\E_{d + 1}(R_{1 + sR}[t])}.
 \end{equation}
The ring $\frac{R_{1 + sR}}{s R_{1 + sR}} = \frac{R}{sR}$ has dimension $d - 1$, which implies that 
$\Um_{d + 1}(\frac{R_{1 + sR}}{sR_{1 + sR}}[t]) =  \E_{d + 1}(\frac{R_{1 + sR}}{sR_{1 + sR}}[t])$. Therefore, the natural map 
 \begin{equation}\label{meq3}
\displaystyle \frac{\Um_{d + 1}(R_{1 + sR}[t], (s))}{\E_{d + 1}(R_{1 + sR}[t], (s))} \longrightarrow \frac{\Um_{d + 1}(R_{1 + sR}[t])}{\E_{d + 1}(R_{1 + sR}[t])}
 \end{equation}
 is surjective, see Remark \ref{retract1}. According to Proposition \ref{RSL}, the orbit space $ \displaystyle \frac{\Um_{d + 1}(R_{1 + sR}[t], (s))}{\E_{d + 1}(R_{1 + sR}[t], (s))}$ is good and $k$-divisible, so by \ref{meq2}, \ref{meq3} and Lemma \ref{pkhom} we deduce that $\displaystyle \frac{\Um_{d + 1}(R_{1 + sR}[t], (t))}{\E_{d + 1}(R_{1 + sR}[t], (t))}$ is good and $k$-divisible. In particular, it follows that $[{\v_0}_{1+sR}]^{r} = [{\v_0}_{1+sR}^{(r)}]$ for $r \in \NN$.

 We observe that $\displaystyle \Psi([\v_0])=([e_1], [{\v_0}_{1+sR}])\in \{[e_1]\}\times  \frac{\Um_{d + 1}(R_{1 + sR}[t], (t))}{\E_{d + 1}(R_{1 + sR}[t], (t))}$ and  $$\Psi([\v_0]^{r})=([e_1], [{\v_0}_{1+sR}]^{r})=([e_1], [{\v_0}_{1+sR}^{(r)}])
 =\Psi([\v_0^{(r)}]) \ \text{for} \  r\in \N.$$
 Since $\Psi$ is injective, we deduce that $[\v_0^{(r)}]=[\v_0]^r$ for $r \in \NN$, therefore $\frac{\Um_{d + 1}(R[t], (t))}{\E_{d + 1}(R[t], (t))}$ is good.
 
 The inverse image $\displaystyle \Psi ^{-1}\Big(\{[e_1]\}\times  \frac{\Um_{d + 1}(R_{1 + sR}[t], (t))}{\E_{d + 1}(R_{1 + sR}[t], (t))}\Big)$ is $k$-divisible and contains $\v_0$, so  we can find $\v_1 \in \Um_{d + 1}(R[t], (t))$ such that $[\v_0] = [\v_1]^k$. Therefore, $\frac{\Um_{d + 1}(R[t], (t))}{\E_{d + 1}(R[t], (t))}$ is $k$-divisible. Hence, the proof is complete.
\end{proof}

As a consequence of the preceding proposition, we establish the following result, which is vital for the proof of the main theorem.

\begin{corollary}\label{mrgrg}
Let $A$ be a positively graded Noetherian ring of dimension $d$ $(\geq 2)$  such that $2k A = A$ for $k \in \NN$. Then $\displaystyle \frac{\Um_{d + 1}(A, A_{+})}{\E_{d + 1}(A, A_{+})}$ is good and $k$-divisible. 
\end{corollary}
\begin{proof}
We set $G = \frac{\Um_{d + 1}(A, A_+)}{\E_{d + 1}(A, A_+)}$ and $H = \frac{\Um_{d + 1}(A[t], (t))}{\E_{d + 1}(A[t], (t))}$. 
The Swan-Weibel map $\psi : A \rightarrow A[t]$ admits the ring homomorphism $\phi_1 : A[t] \xrightarrow{t = 1} A$ such that $\phi_1 \circ \psi = id_A$, see \S \ref{SW}. The sequence of induced homomorphisms :
\[ G \xrightarrow{\psi_*} H \xrightarrow{\phi_{1*}} \frac{\Um_{d + 1}(A)}{\E_{d + 1}(A)}\]
satisfies $\phi_{1*} \circ \psi_* = \alpha$ where $\alpha$ is the natural map $[\v]\mapsto [\v]$ for  $ \v \in \Um_{d+1}(A, A_+
)$. According to Lemma \ref{retract}, there is a split short exact sequence $\displaystyle 0 \rightarrow G \xrightarrow{\alpha} \frac{\Um_{d + 1}(A)}{\E_{d + 1}(A)} \xrightarrow{q_*} \frac{\Um_{d + 1}(R)}{\E_{d + 1}(R)} \rightarrow 0$ of abelian groups, so we obtain a homomorphism $\displaystyle \beta : \frac{\Um_{d + 1}(A)}{\E_{d + 1}(A)} \rightarrow G$ such that $\beta \circ \alpha = id_G$. We deduce $(\beta \circ \phi_{1*}) \circ \psi_* = id_G$, therefore $\psi_*$ is injective and $\beta \circ \phi_{1*}$ is a surjective map.

Let $\v \in \Um_{d + 1}(A, A_+)$. For any $n \in \NN$, the images of $[\v^{(n)}]$ and $[\v]^n$ under $\psi_*$ are equal since $H$ is good according to Proposition \ref{mgkd}.  Now $\psi_*$ is injective, so we deduce that $[\v^{(n)}] = [\v]^n$ for any $n \in \NN$. Therefore, $G$ is good. The orbit space $G$ is $k$-divisible  since $\beta \circ \phi_{1*} : H \rightarrow G$ is a surjective homomorphism and $H$ is  $k$-divisible by Proposition \ref{mgkd}.
\end{proof}

The groundwork laid thus far allows us to prove the main theorem of this article.

\begin{theorem}\label{maint}
    Let $R$ be a local ring of dimension $d$ $(\geq 2)$ , $A$ be a graded $R$-subalgebra of $R[t]$ and $I = \oplus_{i \geq 0} I_i$ be a graded ideal of $A$. Assume $2kR=R$ for some $k \in \NN$. Then $\frac{\Um_{d+1}(A, I)}{\E_{d+1}(A, I)}$ is good and $k$-divisible. 
 \end{theorem}

\begin{proof}
 The orbit space $\frac{\Um_{d+1}(A, I)}{\E_{d+1}(A, I)}$ is an abelian group according to Example \ref{prex}.  By an argument implicit in the proof of  Proposition \ref{42} and by using Lemmas \ref{retract}, \ref{pkhom}, \ref{47}, it suffices to prove the assertion in the case where $A$ is finitely generated and $I = A$. In this setup, let $\v \in \Um_{d + 1}(A)$ and $\m$ be the maximal ideal of $R$. The proof is divided into two steps. 

\noindent

\noindent
{\bf Step- 1: }
Let $B = A/ \m A$. We show that the natural map $\Um_{d + 1}(A) \rightarrow \Um_{d + 1}(B)$ is a surjective map. Let overline denote the image modulo the ideal $\m A$.
We choose $\u = (\overline{u_0}, \ldots, \overline{u_d}) \in \Um_{d + 1}(B)$ for some $(u_0, \ldots, u_d) \in A^{d + 1}$. It is easy to find $u_{d + 1} \in \m A$ such that 
$(u_0, \ldots, u_{d + 1}) \in \Um_{ d + 2}(A)$. We saw in Example \ref{prex} that $\maxSpec (A)$ is a union of two Noetherian spaces of dimension at most $d$, so by  \cite[Proposition 2.4]{van1983group}, there are $c_0, \ldots, c_{d} \in A$ such that $(u_0 + c_0 u_{d + 1}, \ldots, u_d + c_d u_{d + 1}) \in \Um_{d + 1}(A)$. Clearly the image of $(u_0 + c_0 u_{d + 1}, \ldots, u_d + c_d u_{d + 1})$ is $\u$.

\noindent
{\bf Step- 2: }
As before, we assume that overline denotes the image modulo the ideal $\m A$. We observe that $\overline{\v} \in \Um_{d + 1}(B)$. We have $\frac{\Um_{d + 1}(B)}{\E_{d + 1}(B)} \cong \frac{\Um_{d + 1}(B, B_{+})}{\E_{d + 1}(B, B_{+})}$ by Lemma \ref{retract}. Since $4k R = R$, it follows from Corollary \ref{mrgrg} that  $\frac{\Um_{d + 1}(B)}{\E_{d + 1}(B)}$ is good and $2k$-divisible. Using surjectivity of the map in Step- $1$, we find $\w \in \Um_{d + 1}(A)$ such that $[\overline{\v}] = [\overline{\w^{(2k)}}]$ in $\frac{\Um_{d + 1}(B)}{\E_{d + 1}(B)}$. The following sequence is exact by Remark \ref{retract1} and Step-1. 
\begin{equation}\label{maineq1}
\frac{\Um_{d + 1}(A, \m A)}{\E_{d + 1}(A, \m A)} \rightarrow  \frac{\Um_{d + 1}(A)}{\E_{d + 1}(A)}  \rightarrow  \frac{\Um_{d + 1}(B)}{\E_{d + 1}(B)} \rightarrow 0.
\end{equation}
The orbit space $\frac{\Um_{d + 1}(A, \m A)}{\E_{d + 1}(A, \m A)}$ is good and $2k$-divisible by Proposition \ref{maint}. Therefore, we can find $\u \in \Um_{d + 1}(A, \m A)$ such that $[\v]  [ \w^{(2k)}]^{-1} = [\u^{(2k)}]$ which implies   $[\v] = [\u^{(2k)}][ \w^{(2k)}] =([\u^{(2)}][\w^{(2)}])^k$  in $\frac{\Um_{d + 1}(A)}{\E_{d + 1}(A)}$ by (\ref{gop1}) in \S \ref{pregst}. It follows that $\frac{\Um_{d+1}(A)}{\E_{d+1}(A)}$ is $k$-divisible. 

We may assume that $ \u = (a, a_1, a_2, \ldots, a_{d })$, $\w = (b, a_1, a_2, \ldots, a_{d })$ by \cite[Lemma 3.4]{van1983group}. Using (\ref{gop1}) in \S \ref{pregst}, we deduce 
\[
[\v] = [(a^{2k}, a_1, a_2, \ldots, a_{d })][ (b^{2k}, a_1, a_2, \ldots, a_{d })] = [(a^{2k} b^{2k}, a_1, a_2, \ldots, a_{d })],\]
\[[\v]^r = [(a^{2k} b^{2k}, a_1, a_2, \ldots, a_{d })]^r = [(a^{2kr} b^{2kr}, a_1, a_2, \ldots, a_{d })]= [\v^{(r)}],\] for any $r \in \NN$. Therefore, $\frac{\Um_{d+1}(A)}{\E_{d+1}(A)}$ is good and the proof is complete here. 
\end{proof}

\begin{theorem}\label{mainpk2}
  Let $R$ be a ring of dimension $d$ $(\geq 2)$  such that $(d!)R = R$. Let $A = \oplus_{i \geq 0} A_i$ be a  graded $R$-subalgebra of $R[t]$ and $I = \oplus_{i \geq 1} I_i$ be a graded ideal of $A$ generated by homogeneous elements of positive degrees.  Then $\Um_{d + 1}(A, I) = e_1 \SL_{d + 1}(A, I)$.
\end{theorem}

\begin{proof}
We choose $\v \in \Um_{d + 1}(A, I)$. We need to show that $\v \sim_{\SL_{d + 1}(A, I)} e_1$.  By an argument as in Lemma \ref{47}, we find a finitely generated graded $R$-subalgebra $B$ of $A$ such that $B_{+} \subset I$ and $\v \in \Um_{d + 1}(B, B_{+})$. The proof will follow if we show  $\v \sim_{\SL_{d + 1}(B, B_{+})} e_1$.

The local-global principle
enables us to assume that $R$ is a local ring. It follows from Theorem \ref{maint} that there is a $\w \in \Um_{d + 1}(B, B_{+})$ such that $\v \sim_{\E_{d + 1}(B, B_{+})} \w^{(d!)}$. On account of Remark \ref{suscomp}, we have  $\w^{(d!)} \sim_{\SL_{d + 1}(B, B_{+})} e_1$. Hence, $\v \sim_{\SL_{d + 1}(B, B_{+})} e_1$ and this concludes the proof.
\end{proof}
 
 We now present the following example as an application. 
\begin{example}\label{mainex}
Let $R= k[X_1,\ldots,X_d], d \geq 2$ be a polynomial ring over a field $k$ such that $(d !) k = k$. We consider the $R$-algebra homomorphism : 
$$\varphi: R[Y_1,\ldots,Y_d] \rightarrow R[t]\ \text{given by}\ Y_i\mapsto X_i^nt \quad \text{for}\ i=1,2,\ldots,d .$$ Suppose $J = \ker \varphi$, then it follows from \cite[Corollary 5.5.6]{huneke2006integral} that $J$ is generated by maximal minors of the generic matrix
$
 \begin{pmatrix}
X_1^n \cdots X_d^n\\
Y_1 \cdots Y_d
\end{pmatrix}
_{2\times d}$.  Let $I$ be the ideal of $R$ generated by $X_1^n,\ldots,X_d^n$. Then $\Imj \phi = R[It]$ and $R[It] \cong S$, where $\displaystyle S = \frac{k[X_1, \ldots, X_d, Y_1, \ldots, Y_d]}{J}$ is a graded affine algebra of dimension $d + 1$.  Applying Theorem \ref{mainpk2} and \cite[Theorem 2.6]{suslin1977structure}, we deduce $$\Um_{d+1}(S)= e_1 \SL_{d+1}(S).$$

If $n=1$, the algebra $S$ is normal.  If $n\geq 2$, the ideal $I$ is not integrally closed. Consequently, by \cite[Proposition 5.2.1]{huneke2006integral}, the algebra $S$ is not normal. 
\end{example}
\subsection{Questions}
In the course of this work, several questions have arisen that remain unresolved. We list them below to highlight potential directions for future research. The first question has its genesis in the content of \S \ref{pregst}.
\begin{question}
Let $R$ be a ring of dimension $d$. Does the orbit space of $\displaystyle\frac{\Um_{d+1}(R[X_1,\ldots, X_n])}{\E_{d+1}(R[X_1,\ldots, X_n])}$ have a group structure?
\end{question}
We saw that the above question has an affirmative answer when $\dim R = 2$ or $n = 1$. The next question arises naturally in light of Theorem \ref{mainpk2}. 

\begin{question}Let $R$ be a local ring of dimension $d (\geq 2)$ such that $(d!)R=R$ and $A$ be a graded $R$-subalgebra of $R[X_1,\ldots, X_n])$. Is it true that $\Um_{d+1}(A)=e_1\SL_{d+1}(A)$?
\end{question}
More specifically, we are curious to know the following.

\begin{question}Let $R$ be a local ring of dimension $d (\geq 2)$ such that $(d!)R=R$ and $M$ be a positive cancellative monoid. Is it true that $\Um_{d+1}(R[M])=e_1\SL_{d+1}(R[M])$?
\end{question}
An affirmative answer to the above question will be a natural extension of Gubeladze's theorem \cite[Theorem 1.1]{gubeladze2018unimodular}. The final question is inspired by Example \ref{mainex}.
\begin{question} 
Let $k$ be a field and $S = k[X_{ij}, 1 \leq i \leq m, 1 \leq j \leq n]$, $m \leq n$ be a polynomial algebra in $mn$ variables. Suppose $R = \frac{S}{J}$ where $J$ is the ideal generated by the maximal minors of the generic matrix $(X_{ij})_{m \times n}$. Further assume that $((d - 1)!)k = k$ where $d = \dim R = (m - 1)(n + 1)$. Is it true that $\Um_{d}(A)=e_1\SL_{d}(A)$?
\end{question}
If $k$ is algebraically closed and $d \geq 4$, then the above has an affirmative answer  \cite[Theorem 7.5]{fasel2012stably}. 

\begin{acknowledgement}
We express our sincere gratitude to Dr. Sourjya Banerjee for bringing his work to our attention during his visit to Indian Institute of Science Education and Research Bhopal, which helped us formulate Lemma \ref{maindim}. This article is a part of the first author's doctoral dissertation. She is grateful to the National Board of Higher Mathematics for providing financial support, file number 0203/21/2019-RD-II/13102. The second author is thankful to INSPIRE Faculty Fellowship, DST, reference no DST/INSPIRE/04/2018/000522.
\end{acknowledgement}

\bibliographystyle{plain}
\bibliography{reference.bib}
\end{document}